\documentclass[11pt]{amsart}

\usepackage[utf8]{inputenc}
\usepackage[usenames,dvipsnames]{xcolor}
\usepackage{url}
\usepackage{amssymb,pinlabel}

\allowdisplaybreaks

\newcommand{\bbN}{\mathbb{N}}

\newcommand{\bbR}{\mathbb{R}}

\newcommand{\R}{\mathbb R}

\newcommand{\I}{\mathbb I}

\theoremstyle{plain}
\newtheorem{theorem}{Theorem}[section]
\newtheorem{lemma}[theorem]{Lemma}
\newtheorem*{lemma*}{Lemma}
\newtheorem{proposition}[theorem]{Proposition}
\newtheorem{corollary}[theorem]{Corollary}

\newtheorem{question}[theorem]{Question}

\theoremstyle{definition}
\newtheorem{definition}[theorem]{Definition}

\theoremstyle{remark}
\newtheorem{remark}[theorem]{Remark}
\newtheorem{example}[theorem]{Example}

\numberwithin{equation}{section}

\newcommand{\allbf}[1]{\textbf{\boldmath{#1}}}

\newcommand{\twist}{\Delta}

\newcommand{\lsat}[3]{\ensuremath{\Sigma_{#1}(#2,#3)}}

\newcommand{\halfN}{\frac{1}{2}\bbN}

\newcommand{\components}{\mathcal{C}}
\newcommand{\base}{\text{base}}

\newcommand{\figscale}{0.5625}

\usepackage{tikz-cd}

\makeatletter
\newcommand{\xmapsto}[2][]{\ext@arrow 0599{\mapstofill@}{#1}{#2}}
\def\mapstofill@{\arrowfill@{\mapstochar\relbar}\relbar\to}
\makeatother

\begin{document}

\title{Legendrian Satellites and Decomposable Cobordisms}

\author[R.\ Guadagni]{Roberta Guadagni}
\address{Department of Mathematics, University of Pennsylvania} 
\email{robig@sas.upenn.edu}

\author[J.M.\ Sabloff]{Joshua M. Sabloff}
\address{Department of Mathematics and Statistics, Haverford College}
\email{jsabloff@haverford.edu}

\author[M.\ Yacavone]{Matthew Yacavone} 
\address{Galois, Inc.}
\email{matthew@yacavone.net}

\date{\today}



\thanks{When working on this paper, RG was a postdoc at the University of Pennsylvania supported by the Simons' Collaboration HMS grant.  JMS was partially supported by NSF grant DMS-1406093. MY was partially supported by a grant from Haverford College.}

\begin{abstract}
We investigate the interactions between the Legendrian satellite construction and the existence of exact, orientable Lagrangian cobordisms between Legendrian knots.  Given Lagrangian cobordisms between two Legendrian knots and between two Legendrian tangles, we construct a Lagrangian cobordism between Legendrian satellites of the knots by the closures of the tangles, with extra twists on both the top and the bottom satellite to compensate for the genus of the cobordism.  If the original cobordisms were decomposable, then a decomposable cobordism between satellites exists as well, again with extra twists.
\end{abstract}

\maketitle

\section{Introduction}
\label{introduction}

The relation of exact Lagrangian cobordism between Legendrian submanifolds has been a focus of much recent research, with many new obstructions and constructions being developed; see \cite{blllmppst:cob-survey} for a survey.  The goal of this paper is detail a new construction of Lagrangian cobordisms between Legendrian knots using the Legendrian satellite construction. In particular, starting with Lagrangian cobordisms between two Legendrian knots and between two Legendrian tangles, we describe two constructions, one geometric and one combinatorial, of a Lagrangian cobordism between Legendrian satellites of the knots by the closures of the tangles. The construction requires extra twists to compensate for the genus of the cobordism.

In order to state our constructions more precisely and to better motivate the paper, we begin with the formal definition of a Lagrangian cobordism between Legendrians; note that for a subset $A \subset \bbR$, we denote $L \cap (A \times Y)$ by $L_A$.

\begin{definition} \label{defn:lagr-cob}
    Given two Legendrian links $\Lambda_-$ and $\Lambda_+$ in a contact $3$-manifold $(Y, \alpha)$, a \allbf{Lagrangian cobordism} $L$ from $\Lambda_-$ to $\Lambda_+$, denoted $\Lambda_- \prec_L \Lambda_+$, is an exact, orientable, properly embedded Lagrangian submanifold of the symplectization $(\bbR \times Y, d(e^t \alpha))$ that admits a pair of real numbers $T_\pm$ satisfying
    \begin{enumerate}
        \item $L_{(-\infty,T_-]}  = (-\infty,T_-] \times \Lambda_-$ (\textbf{cylindrical lower end}), and
        \item $L_{[T_+,\infty)} = [T_+,\infty) \times \Lambda_+$ (\textbf{cylindrical upper end}). 
            \end{enumerate}
    If $Y$ and $\Lambda_\pm$ have boundary, with $\Lambda_\pm$ properly embedded, then we also require $L$ to be cylindrical in a neighborhood of the boundary.
    
    In the case of disconnected Legendrians, we moreover distinguish between: 
     \begin{itemize}
        \item a \textbf{leveled Lagrangian cobordism}, when the primitive of $e^t\alpha$ along $L$ is constant for $t < T_-$ and for $t> T_+$, and
        \item an \textbf{unleveled Lagrangian cobordism} otherwise.
     \end{itemize}
\end{definition}

\begin{remark}\label{thetaisconstant}
When unspecified, we assume that Lagrangian cobordisms are leveled. This is the standard definition, as it is required, for instance, in order to preserve exactness when composing cobordisms. Note that, if the Legendrians at the ends are connected, any Lagrangian cobordism is leveled, as the primitive is always \emph{locally} constant. For a more thorough discussion of this condition, see \cite{chantraine:disconnected-ends}.
\end{remark}
 
Lagrangian cobordisms induce a relation on the set of isotopy classes of Legendrian knots.  This relation's basic properties differ from those of the relation induced by smooth cobordism.  Let us consider three subtleties.  First, while smooth cobordism induces an equivalence relation on smooth knots, the Lagrangian cobordism relation between Legendrians is not symmetric  \cite{bs:monopole,chantraine:non-symm, cns:obstr-concordance}.  Second, while the smooth collar neighborhood theorem allows any two smooth cobordisms $S_1$ from $K_0$ to $K_1$ and $S_2$ from $K_1$ to $K_2$ to be stacked one on top of another to form a new cobordism $S_1 \odot S_2$,  unleveled cobordisms cannot always be stacked \cite{chantraine:collar}, hence we require all  Lagrangian cobordisms to be leveled when we want the cobordism relation to be transitive. Finally, and of particular interest for this paper, Morse theory tells us that a generic smooth cobordism may  be decomposed into a stack of elementary cobordisms arising from smooth isotopy or handle attachment. The decomposability of Lagrangian cobordisms is more delicate.  On one hand, there are  constructions of elementary Lagrangian cobordisms corresponding to isotopy and the attachment of $0$- and $1$-handles (see Section~\ref{cobordismsBkg}, and particularly Figure~\ref{fig:construct}, for details).  Given a sequence of elementary Lagrangian cobordisms $\Lambda_0 \prec_{E_1} \Lambda_1, \ldots, \Lambda_{k-1} \prec_{E_k} \Lambda_k$, we say that the Lagrangian cobordism $\mathbf{L} = E_1 \odot \cdots \odot E_k$ is \allbf{decomposable}.  On the other hand, the following question is open as of this writing:

\begin{question} \label{q:main}
    If a pair $\Lambda_\pm$ of non-empty, non-stabilized Legendrian links is Lagrangian cobordant, does there exist a decomposable cobordism between them?
\end{question}

Note that the stipulation of non-empty, non-stabilized Legendrians at the ends is necessary, as there exist non-decomposable Lagrangian caps $\Lambda \prec_L \emptyset$ when $\Lambda$ is sufficiently stabilized \cite{lin:lagr-cap}.  See \cite{blllmppst:cob-survey} for further discussion.

One possible source of Lagrangian cobordisms without decomposable analogues arises from the Legendrian satellite construction.  We briefly set notation, delaying precise definitions to Section~\ref{background}.  Given a Legendrian knot $\Lambda$ and a Legendrian $n$-tangle $\Pi$ in $J^1[0,1]$ with closure $\phi(\Pi)$ in $J^1S^1$, we denote the Legendrian satellite of $\Lambda$ by $\phi(\Pi)$ with the contact (Thurston–Bennequin) framing by $\lsat{}{\Lambda}{\Pi}$. Given a Lagrangian \emph{concordance} from $\Lambda_-$ to $\Lambda_+$ and an $n$-stranded pattern $\Pi$, one can find, via a geometric construction in \cite[Theorem 2.4]{cns:obstr-concordance}, a Lagrangian concordance from $\lsat{}{\Lambda_-}{\Pi}$ to $\lsat{}{\Lambda_+}{\Pi}$. It is not at all clear that the result of this construction is decomposable even if the original Lagrangian is.  In fact,  Cornwell, Ng, and Sivek \cite[Conjecture 3.3]{cns:obstr-concordance} construct a candidate for concordant pair of Legendrians that are not decomposably concordant by taking the Whitehead double of a concordance from the maximal Legendrian unknot to a Legendrian $m(9_{46})$ knot.

The two main results of this paper generalize the construction of Cornwell, Ng, and Sivek to Lagrangian cobordisms of arbitrary topology and supply sufficient conditions for the existence of a decomposable cobordism between satellites. Both results use the full twist Legendrian tangle in $\twist \subset J^1[0,1]$, pictured in Figure~\ref{fig:twist}.

The first main theorem arises from a geometric construction.

\begin{figure}
    \centering
    \includegraphics[scale=0.75]{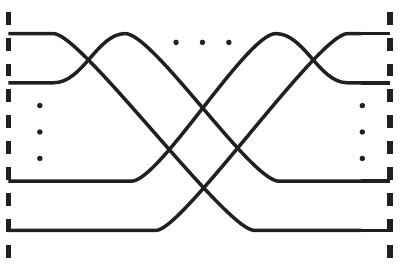} \hspace{2em}
    \includegraphics[scale=0.75]{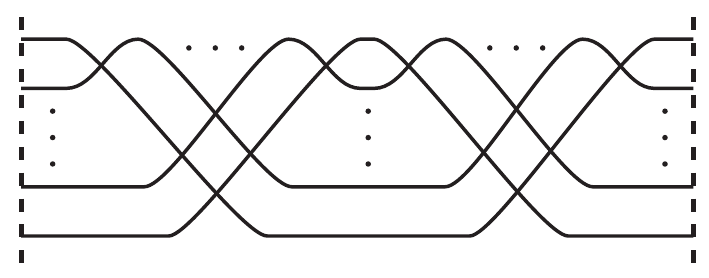}
    \caption{On the left, the Legendrian $n$-stranded half twist $\twist^{1/2}$ with $n(n-1)/2$ crossings. On the right, the Legendrian full twist $\twist$ as the composition of two half-twists.}
    \label{fig:twist}
\end{figure}

\begin{theorem}\label{thm:geom}
    Suppose $\Lambda_\pm$ are non-empty Legendrian knots in a contact manifold $(Y, \alpha)$ and $\Pi_\pm$ are Legendrian $n$-tangles in the standard $J^1[0,1]$. Given Lagrangian cobordisms $\Lambda_- \prec_L \Lambda_+$ and $\Pi_- \prec_P \Pi_+$, there exists a Lagrangian cobordism
        \[ \lsat{}{\Lambda_-}{\twist^{2g(L)+1}\Pi_-} \prec \lsat{}{\Lambda_+}{\Delta\Pi_+}, \]
and a possibly non-orientable Lagrangian cobordism
\[ \lsat{}{\Lambda_-}{\Delta \Pi_-} \prec \lsat{}{\Lambda_+}{\Delta \Pi_+}. \]
\end{theorem}

The second main theorem provides a decomposable counterpart of the first for Legendrians in the standard contact $\bbR^3$. To state the second theorem, we require the following technical property:

\begin{definition}[Property A]
\label{def:property-A}
A decomposable cobordism $\mathbf{L}$ satisfies \allbf{Property A} if for every $1$-handle $E_i$ in $\mathbf{L}$ that increases the number of components from $\Lambda_{i-1}$ to $\Lambda_i$,
there exists a path in $\mathbf{L}$ which starts at the base $\Lambda_{i-1}$ of $E_i$, ends at some component of $\Lambda_-$, and whose symplectization coordinate is monotone decreasing.
\end{definition}

In many cases one can deform $\mathbf{L}$ to ensure that it satisfies Property A, but the authors are unsure if this can always be done.  Note that if $\mathbf{L}$ is a concordance or has no 0-handles, then $\mathbf{L}$ always satisfies Property A.

\begin{theorem} \label{thm:decomp}
      Suppose $\Lambda_\pm$ are non-empty Legendrian knots in $(\bbR^3, \alpha_0)$ and $\Pi_\pm$ are Legendrian $n$-tangles in the standard $J^1[0,1]$.  Given  decomposable Lagrangian cobordisms $\Lambda_- \prec_\mathbf{L} \Lambda_+$ and $\Pi_- \prec_\mathbf{P} \Pi_+$ such that $\mathbf{L}$ satisfies Property A, there exists a  decomposable Lagrangian cobordism
 \[ \lsat{}{\Lambda_-}{\twist^{2g(\mathbf{L})+1}\Pi_-} \prec
               \lsat{}{\Lambda_+}{\twist\Pi_+}, \]
and (regardless of whether $\mathbf{L}$ satisfies Property A) a possibly non-orientable decomposable Lagrangian cobordism
\[ \lsat{}{\Lambda_-}{\twist^{1/2} \Pi_-} \prec \lsat{}{\Lambda_+}{\twist^{1/2} \Pi_+}. \]
\end{theorem}

See Proposition~\ref{prop:main_geom}, Theorem~\ref{thm:decompTwist}, and Corollary~\ref{cor:decompGen} for examples of more refined conditions under which Lagrangian cobordisms --- decomposable or otherwise --- exist between satellites.

\begin{remark}
Although the two main theorems are semantically similar, the cobordisms constructed in their proofs may be quite different. 
\end{remark}

\begin{remark}
The orientable and possibly non-orientable cobordisms in each theorem have related constructions, but are quite different geometrically.  For example, the orientable cobordisms in both main theorems lie in a $C^1$ neighborhood of the original cobordism $L$, while the possibly non-orientable cobordisms may not even lie in a $C^0$ neighborhood of $L$.
\end{remark}

With Theorems~\ref{thm:geom} and \ref{thm:decomp} in hand, one can approach Question~\ref{q:main} by trying to find a pair of non-stabilized Legendrians $\Lambda_\pm$ that are Lagrangian cobordant but not decomposably cobordant.  A first step would involve a search for a decomposable cobordism $L$ that does \emph{not} satisfy Property $A$ and is not isotopic to a cobordism satisfying Property $A$.  One would then use Theorem~\ref{thm:geom} to construct a satellite of the cobordism $L$; the resulting cobordism would be a higher-genus candidate than the aforementioned construction in \cite{cns:obstr-concordance}.

The remainder of the paper is organized as follows. We review necessary background on Legendrian tangles, the Legendrian satellite construction, and Lagrangian cobordisms between Legendrians in Section~\ref{background}. We then proceed to prove Theorem~\ref{thm:geom} in Section~\ref{geom} (by proving Proposition \ref{prop:main_geom} and Proposition \ref{main_geom_nonorientable}). We prove Theorem~\ref{thm:decomp} in Section~\ref{decomp} as a consequence of the more general Theorem~\ref{thm:decompTwist}.  In Section~\ref{sec:connect-sum}, we conclude with some remarks about connected sums and generalizations of the results in this paper to higher dimensions.

\subsection*{Acknowledgements}

The authors thank John Etnyre, Yanhan Liu, Sipeng Zhou, Baptiste Chantraine, and the participants of the Philadelphia Area Contact / Topology seminar for stimulating discussions about the material herein. The authors also thank the referee for their constructive and valuable suggestions.

\section{Background and Foundations}
\label{background}

In this section, we describe the three central objects in this paper in more detail:  Legendrian $n$-tangles, Legendrian satellites, and decomposable Lagrangian cobordisms.  We assume familiarity with basic notions of Legendrian knot theory (as in \cite{etnyre:knot-intro, geiges:intro}) and of Lagrangian submanifolds (as in \cite{audin-lalonde-polterovich}).

\subsection{Legendrian \textit{n}-Tangles} 
\label{solid-torusBkg}

The canonical solid cylinder in contact topology is the $1$-jet space $J^1[0,1] = [0,1] \times \bbR^2$, which has coordinates $(x, y, z)$ and contact form $dz-y\,dx$. For $n > 0$, a \allbf{Legendrian $n$-tangle $\Pi$} is a proper Legendrian embedding of $n$ disjoint intervals and a finite number of disjoint copies of $S^1$ into $J^1[0,1]$ 
such that the front projections of the intervals are horizontal near the boundary and $\Pi\cap \left.J^1[0,1]\right|_{x=i} = \bigl\{(i,0,\frac{k}{n+1})\;:\; 1 \leq k \leq n\bigr\}$ for $i=0,1$.

An \textbf{orientation} on an $n$-tangle is an orientation of each of the components such that the orientation at $J^1\{0\}$ matches that at $J^1\{1\}$. For any orientation $s$, let $-s$ be its reversed orientation and let $\overline{s}$ be defined as follows: the value of $\overline{s}$ on the $i^{\text{th}}$ strand at $J^1\{0\}$ is the value of $s$ on the $(n+1-i)^{\text{th}}$ strand at $J^1\{0\}$, and  on any disjoint $S^1$, the value of $\overline{s}$ is equal the value of $s$. An orientation $s$ is \textbf{symmetric} if $s = \overline{s}$ and \textbf{uniform} if the value of $s$ agrees on all strands at $J^1\{0\}$.

\begin{example}
 While the full twist $\twist_s$ pictured in Figure~\ref{fig:twist} is an oriented tangle for any choice of orientation $s$, the half-twist $\twist^{1/2}_s$ is an oriented tangle only when $s$ is symmetric.
\end{example}

The canonical quotient map $[0,1] \to S^1$ yields a  Legendrian solid torus link $\phi(\Pi) \subseteq J^1S^1$ for every Legendrian $n$-tangle $\Pi$; if the tangle is oriented, then the torus link inherits the orientation. The link $\phi(\Pi)$ is well-defined up to Legendrian isotopy \cite{lenny-lisa}.

The operation of adding a full twist to an existing Legendrian $n$-tangle plays a key role in the statement of Theorems \ref{thm:geom} and \ref{thm:decomp}.   Generally speaking, if $\Pi_1$ and $\Pi_2$ are both Legendrian $n$-tangles, then we can form a new Legendrian $n$-tangle $\Pi_1 \Pi_2$ by identifying $\Pi_1\cap\{1\}$ with $\Pi_2\cap\{0\}$ and scaling the $x$ and $y$ directions appropriately. We will be particularly interested in products of the form $\Delta^k\Pi$.  

\begin{remark}
The ability to use these products is the reason we work with Legendrian $n$-tangles instead of directly with Legendrian solid torus links in defining satellites:  as noted in \cite[Remark 4.2]{ev:satellites}, adding a positive twist to a Legendrian solid torus link is not a well-defined operation.  
\end{remark}

\begin{remark}\label{clockwise_twist}
The twist $\Delta$ is a clockwise rotation when moving left to right; a counterclockwise twist would require cusps in the front diagram, and hence would not project $n$-to-$1$ to the base over every point.
\end{remark}

\subsection{Legendrian Satellites}
\label{satellitesBkg}

To define the Legendrian satellite of a Legendrian knot $\Lambda \subset (Y, \alpha)$  by a Legendrian $n$-tangle $\Pi \subset J^1[0,1]$, we  follow \cite{ev:satellites,ng:satellite, lenny-lisa}. It is a fundamental fact that $\Lambda$ has a canonical neighborhood that is contactomorphic to $J^1S^1$.  We define the \allbf{Legendrian satellite} of $\Lambda$ by $\Pi$, denoted $\lsat{}{\Lambda}{\Pi}$, to be the link that results from removing a standard neighborhood of $\Lambda$ and gluing in the solid torus containing $\phi(\Pi)$. Note that the satellite operation is well-defined \cite{lenny-lisa}. There is an obvious generalization to satellites of $r$-component Legendrian links by $r$-tuples of Legendrian $n$-tangles. Note that, unlike in the smooth case, the Legendrian satellite has a built-in framing governed by the Thurston--Bennequin number of $\Lambda$.  

At a diagrammatic level, Figure~\ref{fig:front-sat} depicts a procedure for drawing a front diagram of $\lsat{}{\Lambda}{\Pi}$ from front diagrams of $\Lambda$ and of $\Pi$.  Suppose that $\Pi$ is a Legendrian $n$-tangle. Start by taking the $n$-copy of a front diagram of $\Lambda$, i.e. $n$ copies of $\Lambda$ that differ by a small vertical shift. Replace the trivial $n$-tangle $\I_n$ created by the $n$-copy of a small interval in the front of $\Lambda$ (away from cusps and crossings) with the front diagram of $\Pi$.  The result is a front diagram for $\lsat{}{\Lambda}{\Pi}$.

\begin{figure}
\labellist
    \small
    \pinlabel $\Lambda$ [l] at 140 120
    \pinlabel $\Pi$ [l] at 140 10
    \pinlabel $\lsat{}{\Lambda}{\Pi}$ [l] at 380 10
\endlabellist

    \centering
    \includegraphics[width=4.5in]{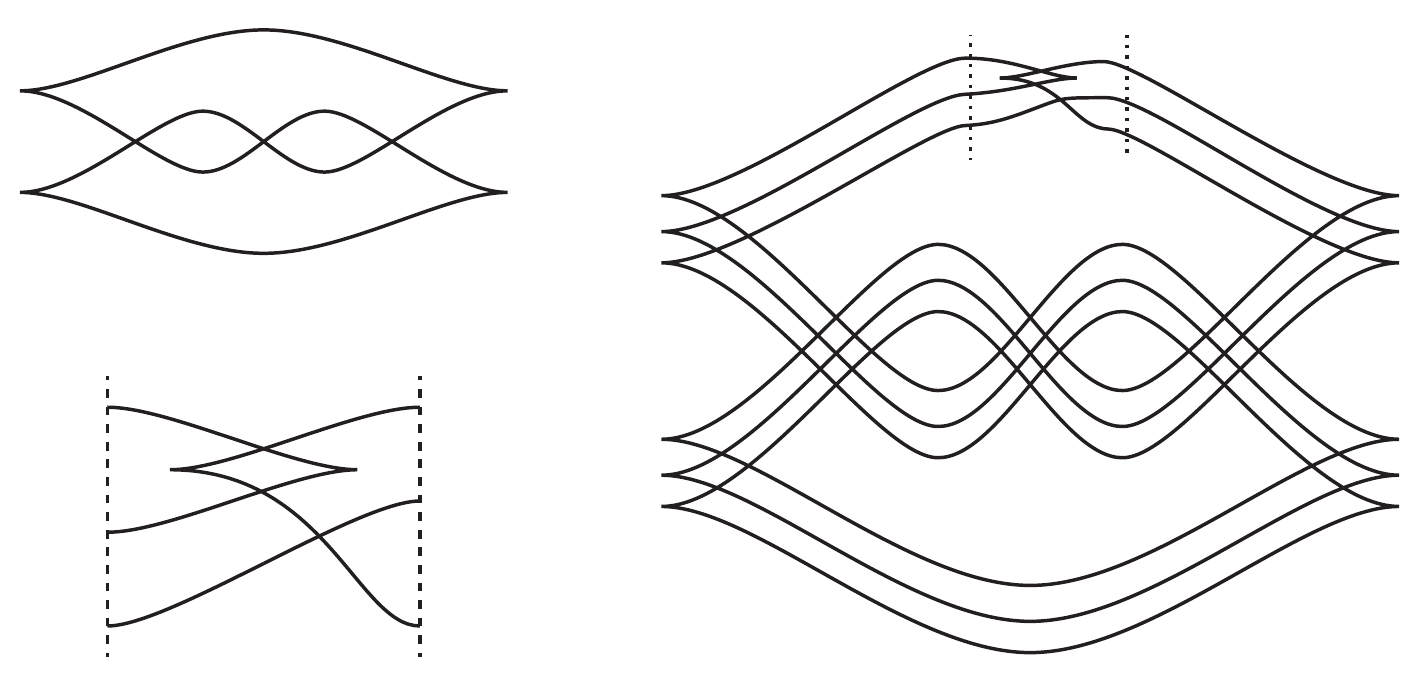}
    \caption{A procedure for creating a front diagram for $\lsat{}{\Lambda}{\Pi}$.} 
    \label{fig:front-sat}
\end{figure}

\subsection{Decomposable Cobordisms}
\label{cobordismsBkg}

A Lagrangian cobordism is \allbf{elementary} if it arises from one of the  cobordisms defined by the following theorem. 

\begin{theorem}[\cite{bst:construct, rizell:surgery, ehk:leg-knot-lagr-cob}] \label{thm:construct}
    If two Legendrian links $\Lambda_-$ and $\Lambda_+$ are related by any of the following moves, then there exists a Lagrangian cobordism $\Lambda_- \prec_L \Lambda_+$.
    \begin{description}
        \item[Isotopy] $\Lambda_-$ is Legendrian isotopic to $\Lambda_+$.
        \item[$0$-handle] The front diagrams for $\Lambda_-$ and $\Lambda_+$ are identical except for the addition of a disjoint maximal Legendrian unknot $\Upsilon$ in $\Lambda_+$ as in Figure~\ref{fig:construct} (left).
        \item[$1$-handle] The front diagrams for $\Lambda_-$ and $\Lambda_+$ are related as in Figure~\ref{fig:construct} (right).
    \end{description}
\end{theorem}

\begin{figure}
\labellist
    \large
    \pinlabel $\emptyset$ [b] at 32 25
\endlabellist

    \centering
    \includegraphics[height=2in]{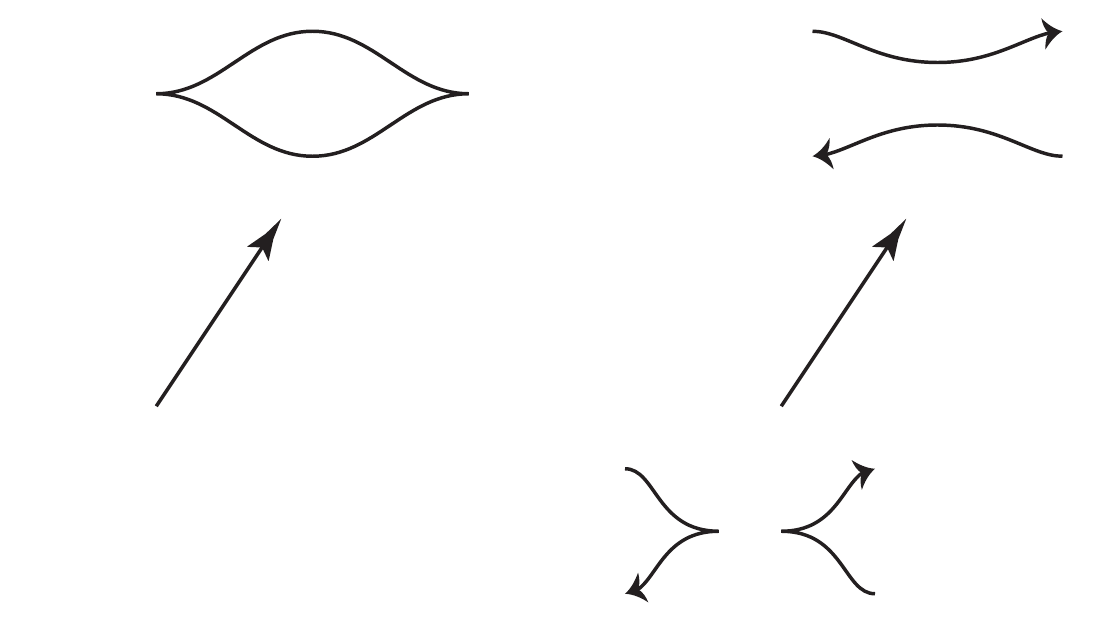}
    \caption{Diagrammatic moves corresponding to the attachment of a $0$-handle (left) and a $1$-handle (right).}
    \label{fig:construct}
\end{figure}

A $1$-handle cobordism which joins two different components in $\Lambda_-$, and hence reduces the number of components by one in $\Lambda_+$, we call a \allbf{$\wedge$-handle}. A $1$-handle which splits a component in $\Lambda_-$, and hence increases the number of components by one in $\Lambda_+$, we call a \allbf{$\vee$-handle}. Since our $1$-handles are orientable, these are the only two possibilities. For any $\vee$-handle $E$, let $\base(E)$ be the component of $\Lambda_-$ on which $E$ acts, and for any $\wedge$-handle $E$, let $\base_1(E)$ and $\base_2(E)$ be the two distinct components of $\Lambda_-$ on which $E$ acts.  This language will become important in Section~\ref{decomp}.

A \allbf{decomposable} Lagrangian cobordism $\mathbf{L}$ is the result of composing a sequence $\Lambda_0 \prec_{E_1} \Lambda_1, \ldots, \Lambda_{k-1} \prec_{E_k} \Lambda_k$ of elementary Lagrangian cobordisms:
\[ \mathbf{L} =   E_1 \odot \cdots \odot E_k.\]

As the constructions in Theorem~\ref{thm:construct} are all local, the same statement and vocabulary may be used to define elementary and decomposable Lagrangian cobordisms between Legendrian $n$-tangles.  Further, note that elementary cobordisms preserve the orientations of $n$-tangles at the boundary.

\section{Satellite construction for Lagrangian cobordisms}
\label{geom}

The goal of this section is to geometrically construct Lagrangian cobordisms between satellites of Legendrian links, which is embodied in the following  proposition.  

\begin{proposition}
\label{prop:main_geom}
Suppose $\Lambda_\pm$ are $r_\pm$-component Legendrian links in $(Y,\alpha)$ and $\Pi_\pm$ are Legendrian $n$-tangles. Given a connected Lagrangian cobordism $\Lambda_-\prec_L\Lambda_+$ and a Lagrangian cobordism $\Pi_-\prec_P\Pi_+$, there exists an orientable Lagrangian cobordism $\Sigma(L,P)$ from the satellite of $\Lambda_-$ by any $r_-$-tuple of tangles $(\Delta^{1+2g_1+k_1}\Pi_-, \Delta^{1+2g_2+k_2}, \ldots, \Delta^{1+2g_{r_-}+k_{r_-}})$ to the satellite of $\Lambda_+$ by any $r_+$-tuple of tangles $(\Delta^{l_1} \Pi_+, \Delta^{l_2}, \ldots, \Delta^{l_{r_+}})$ for any choice of $g_i, k_i, l_i$ so that $\sum g_i = g(L)$, $\sum k_i = r_- -1$, and $\sum l_i = 1$.
\end{proposition}

The first claim of Theorem~\ref{thm:geom} follows when the Legendrians at the ends are connected. Proposition~\ref{prop:main_geom} may also be applied to construct the possibly non-orientable cobordisms in the second claim of Theorem~\ref{thm:geom}.

\begin{proposition}
\label{main_geom_nonorientable}
Suppose that $\Lambda_\pm$ are Legendrian knots in $(Y,\alpha)$ and $\Pi_\pm$ are Legendrian $n$-tangles.  Given cobordisms $\Lambda_-\prec_L\Lambda_+$ and $\Pi_-\prec_P\Pi_+$, there exists a (possibly non-orientable) cobordism \[\Sigma(\Lambda_-, \Delta \Pi_-)\prec_{\tilde{\Sigma}(L,P)} \Sigma(\Lambda_+,\Delta \Pi_+).\]
\end{proposition}

The proof of Proposition \ref{prop:main_geom} requires four main steps, which shall be worked out in the subsections below:
\begin{enumerate}
\item Find an appropriate tubular neighborhood of $L$ that is exact symplectomorphic to $T^*L$;
\item Develop conditions on functions $f: L \to \bbR$ so that the graph of $df$ in $T^*L$ yields a leveled Lagrangian cobordism;
\item Construct a non-singular function $f$ that satisfies properties necessary to produce an $n$-copy of $L$; and 
\item Sew in the cobordism $P$.
\end{enumerate}

\subsection{Tubular neighborhoods of Lagrangian cobordisms}
\label{ssec:tubular}

The classical Weinstein Neighborhood Theorem states that a compact Lagrangian $L$ has a tubular neighborhood that is symplectomorphic to a neighborhood of the zero section in $T^*L$.  We need to extend the theorem to (non-compact) Lagrangian \emph{cobordisms} and \emph{exact} symplectomorphisms. Recall that a symplectomorphism $\phi:(M_1,d\alpha_1)\to (M_2,d\alpha_2)$ is exact when $\phi^*(\alpha_2)=\alpha_1+d\theta$ for some function $\theta$ on $M_1$. Though such an extension of the Weinstein Neighborhood Theorem has been assumed in the literature in, for example, \cite{c-dr-g-g-cobordism, josh-lisa:rel-gr-width}, we will discuss a proof of this folklore theorem here. Our discussion has the added benefit of introducing notation that will be useful in later sections. We begin by examining the exactness of the symplectic embedding in the Weinstein Neighborhood Theorem, which will rely on a version of Moser's Method.

\begin{lemma}[Moser's Method for exact symplectic forms, cf. Theorem 6.8 of \cite{ce:weinstein}] \label{lem:exactmoser}
Given a family of exact symplectic forms $\omega_t = d\alpha_t$ on a  manifold $V$ without boundary that coincide outside a compact set, there exists a family of diffeomorphisms $\phi_t: V \to V$  with compact support such that $\phi_t^*\alpha_t=\alpha_0+d\theta_t$.  In particular, $\phi_1: (V, d\alpha_0) \to (V, d\alpha_1)$ is an exact symplectomorphism.
\end{lemma}

In the case of a closed exact Lagrangian, the folklore theorem that the symplectic embedding can be made exact can be proven by applying the version of Moser's Method in \cite[Lemma 11.2]{ce:weinstein} to the proof of the Weinstein Neighborhood Theorem in \cite[\S6.4]{ce:weinstein}.  In particular, if $L\subset(M,\omega=d\alpha)$ is a compact, exact Lagrangian submanifold, then there exists an exact symplectic embedding $\phi: U \subset T^*L\hookrightarrow M$, defined on a neighborhood $U$ of the zero section.

To understand neighborhoods of Lagrangian cobordisms, we first turn our attention to the cylindrical ends $\bbR \times \Lambda$ in the symplectization $(\bbR \times Y, d(e^t\alpha))$.  A neighborhood of a Legendrian $\Lambda$ in $(Y, \alpha)$ is naturally modelled on $J^1\Lambda$, while a neighborhood of $L=\bbR \times \Lambda$ is naturally modelled on $T^*L$. There are a variety of ways to identify the symplectization of a $1$-jet space with the cotangent bundle of a cylindrical end (see, for example, \cite{c-dr-g-g-cobordism, josh-lisa:obstr}); we use an identification that has the feature of yielding a morphism of symplectic bundles over $L$. On one hand, let $J^1\Lambda = \bbR \times T^*\Lambda$ have local coordinates $(v,s,u)$, and let the symplectization coordinate in $\bbR \times J^1 \Lambda$ be denoted by $t$; the primitive of the symplectic form is $e^t \alpha = e^t(dv-u\,ds)$.  On the other hand, consider the exact symplectic space $T^*(\bbR \times \Lambda) = T^*\bbR \times T^*\Lambda$ with coordinates $(t,v,s,u)$ whose primitive is the tautological $1$-form $-\lambda = -(v\,dt + u\, ds)$.  We now define $\beta: \bbR \times J^1 \Lambda \to T^*(\bbR \times \Lambda)$ by
\[ \beta(t,v,s,u) = (t, e^tv, s, e^t u).\]
It is straightforward to check that $\beta$ is an exact symplectomorphism with $\beta^*(-\lambda) = e^t \alpha - d(e^t v).$

We can now impose a natural compatibility condition between a Lagrangian tubular neighborhood of $\bbR \times \Lambda$ and a Legendrian tubular neighborhood of $\Lambda$.  If $\phi: U \subset T^*(\bbR \times \Lambda) \hookrightarrow \bbR \times Y$ is an exact symplectic embedding of a neighborhood of the zero section and $\psi: V \subset J^1\Lambda \hookrightarrow Y$ is a contact embedding of a neighborhood of the zero section, we say that $\phi$ and $\psi$ are \allbf{compatible} if $id \times \psi = \phi \circ \beta$, perhaps after shrinking domains.

With the notion of compatible embeddings of cylindrical cobordisms in hand, we are ready to state the Weinstein Neighborhood Theorem for Lagrangian cobordisms.

\begin{lemma}\label{tubnbhd} 
Given an exact Lagrangian cobordism $\Lambda_- \prec_{L}\Lambda_+$ and Legendrian tubular neighborhoods of $\Lambda_\pm$ defined by embeddings $\psi_\pm$, there exists an exact symplectic embedding $\phi: U \subset T^* L \hookrightarrow \R\times Y$ of a neighborhood of the zero section that is compatible with $\psi_\pm$
in the cylindrical ends.
\end{lemma}

\begin{proof}
On the  cylindrical ends, the embeddings $\psi_\pm$ define compatible exact symplectic embeddings of the zero sections of the ends of $L$ via 
\[\phi_\pm = (id \times \psi_\pm)\circ \beta^{-1}.\]
We next extend $\phi_\pm$ to an embedding $\tilde{\phi}: U \subset T^*L \to \bbR \times Y$ of a neighborhood of the zero section. We could do so, for example, by realizing $\phi_\pm$ as exponential maps of metrics $g_\pm$ compatible with the symplectic form, then patching $g_\pm$ together with a compatible metric on the remaining part of the symplectization to get a compatible metric on all of $\bbR \times Y$, and finally using the exponential map along $L$.

Since $\tilde{\phi} $ was already an exact symplectomorphism on the cylindrical ends, we can now use (the exact version of) Moser's method as in Lemma~\ref{lem:exactmoser} to modify $\tilde\phi$ along a compact subset of $L$ to an exact symplectic embedding $\phi:T^*L\hookrightarrow \R\times Y$, compatible with $\psi_\pm$ at the ends.
\end{proof}

\subsection{Deformations of Lagrangian cobordisms}

The next step in the construction in Proposition~\ref{prop:main_geom} is to create deformations of Lagrangian cobordisms $\Lambda_- \prec_L \Lambda_+$ using small Hamiltonian diffeomorphisms inside $T^*L$.  More specifically, consider a smooth function $f: L \to \bbR$; the graph $\Gamma(df)$ is an exact Lagrangian in $T^*L$.  If $\Gamma(df)$ lies in the domain of an exact Lagrangian tubular neighborhood, it is straightforward to check that its image $L^f$ in $\bbR \times Y$ is exact Lagrangian if and only if $L$ is.  

The remaining step to ensure that $L^f$ is a leveled Lagrangian cobordism is to control the behavior of $f$ at the ends.  In this subsection, we analyze how various behaviors of $f$ at the upper and lower ends imply properties of $L^f$.

\begin{definition}
Suppose $L = I \times \Lambda$ is a connected and cylindrical Lagrangian submanifold of the symplectization $\bbR \times Y$ for an interval $I \subset \bbR$.  Let $L$ be parametrized by coordinates $(t,s) \in I \times \Lambda$. A function $f: L\to \R$ is \textbf{cylindrical} if it is of the form $f(t,s)=e^tg(s)+C$. 
\end{definition}

When $L$ is cylindrical but disconnected, we can ask that $f$ be cylindrical on each connected component. Notice that the additive constant $C$ might then be different on different connected components, thus making the function \emph{unleveled}. To avoid confusion, we provide the following terminology. 

\begin{definition}
For a Lagrangian cobordism $L$, a function $f: L\to \R$ has a \textbf{cylindrical upper (resp. lower) end} if it is cylindrical on the upper (resp. lower) end of $L$. If, moreover, the additive constant $C$ is the same on all connected components of the end, we say that the function has a \textbf{leveled cylindrical upper (resp. lower) end}.
\end{definition}

\begin{definition}\label{defn:glob-cyl}
A function $f$ is \textbf{globally cylindrical} when it has leveled cylindrical upper and lower ends and there exists an embedded strip $\Phi: \bbR \times (-\epsilon,\epsilon) \hookrightarrow L$ such that 
\begin{enumerate}
\item The embedding is of the form $\Phi(u,v) = (u,\phi_{\pm}(v))$ outside of $[T_-, T_+] \times (-\epsilon,\epsilon)$ and
\item  Along the strip, we have $f \circ \Phi (u,v) = e^u+C$. 
\end{enumerate}
\end{definition}

\begin{remark}
When $f$ is globally cylindrical, there is a unique additive constant $C$ along all components of both the lower and upper ends.
\end{remark}

\begin{example}
If $f: L_{[T_-,T_+]} \to [b_-, b_+]$ is a Morse function --- that is, $f$ has non-degenerate critical points on the interior and $f^{-1}(b_\pm) = L_{\{T_\pm\}}$ (cf.\ \cite{milnor:h-cob}) --- then, up to a perturbation that does not change the critical points of $f$, we may assume that $f(s,t) = k_\pm e^t$ in neighborhoods of $T_\pm$. Therefore we can extend $f$ to a function $f: L\to\R$ that has leveled cylindrical upper and lower ends, with $C_+=C_-=0$.

If we start with a Morse function $f: L_{[T_-,T_+]} \to [e^{T_-}, e^{T_+}]$, we have that $k_\pm=1$. Consider a flowline $\gamma$ of $df$ from $\Lambda_-$ to $\Lambda_+$. It is easy to find a neighborhood of $\gamma$ that has the form $\Phi=\gamma\times (-\epsilon,\epsilon)$ and we can assume without loss of generality that $f=e^t$ in local coordinates. In other words, $f$ is globally cylindrical.
\end{example}

The following proposition clarifies the usefulness of our definitions.

\begin{proposition} \label{prop:displaced-cobordism}
Given a leveled Lagrangian cobordism $\Lambda_- \prec_L \Lambda_+$, 
\begin{enumerate}
\item $L^f$ is a Lagrangian cobordism if and only if $f:L\to\R$ has cylindrical upper and lower ends, and $L^f$ is leveled if and only if $f$ is; and
\item 
$L\cup L^f$ is a leveled Lagrangian cobordism from $\Lambda_-\cup \Lambda^f_-$ to $\Lambda_+ \cup \Lambda^f_+$ if and only if $f:L\to\R$ is nonsingular and has leveled cylindrical ends with $C_+=C_-$.
\end{enumerate}
\end{proposition}

We first prove that $f$ is cylindrical at the ends if and only if $L^f$ is. We will then verify that the cobordism we obtain is leveled, for which we need the extra assumption that the function is leveled cylindrical. These facts are proven in Lemmas~\ref{lem:cylindrical_conditions} and Lemma~\ref{lem:disconnected_ends}, respectively. Note that the requirement that $f$ be nonsingular in part (2) of the proposition above guarantees that $L \cup L^f$ is embedded.

\begin{lemma}\label{lem:cylindrical_conditions}
For a cylindrical Lagrangian $L = I \times \Lambda$, the following are equivalent for a function $f: L \to \bbR$:
\begin{enumerate}
\item The function $f$ is cylindrical; and
\item The Lagrangian $L^f$ is cylindrical.
\end{enumerate}
 Moreover, if these conditions hold, the  Legendrian always takes the form $\Lambda^f = J^1(g)$ in $J^1 \Lambda$.
\end{lemma}

\begin{proof}
The proof that the first condition implies the second follows from a straightforward calculation using the symplectomorphism $\beta$ defined in the previous section, as does the last statement of the lemma.

For the reverse direction, notice that we may parametrize $\beta^{-1}(\Gamma(df))$ in $\bbR \times J^1 \Lambda$ by $\left(t,s,e^{-t}\partial_t f, e^{-t} \partial_s f\right).$ The assumption that $L^f$ is cylindrical implies that a displacement in the symplectization ($t$) direction preserves $L^f$.  In particular, we have \[\partial_t \left((t,s,e^{-t}\partial_t f,e^{-t}\partial_sf)\right)=(1,0,0,0).\]
The last two components of this equation imply that $\partial_{tt} f =\partial_t f$ and $\partial_{ts} f=\partial_s f $, respectively. Solving for $f$ using the first and then the second equation yields $f=e^tg(s)+C$ for some $g(s)$ and $C$ on each connected component, as desired.
\end{proof}

\begin{lemma}\label{lem:disconnected_ends} 
Let $L$ be a leveled Lagrangian cobordism and let $f: L \to \bbR$ be a cylindrical function.  
\begin{enumerate}
\item The cobordism $L^f$ is leveled if and only if $f$ has leveled cylindrical lower and upper ends.  
\item Moreover, the cobordism $L \cup L^f$ is leveled if and only if the additive constant for $f$ is the same at the upper and lower ends.
\end{enumerate}
\end{lemma}

\begin{proof}
	We begin by setting notation.	Denote by $\theta$ (resp.\ $\theta^f$) the primitive of $e^t \alpha$ along $L$ (resp.\ $L^f$).  Denote by $\phi_\tau$ the time-$1$ Hamiltonian isotopy that displaces $L$ to $L^f$, with the Hamiltonian function $f$ and Hamiltonian vector field $X$.  Finally, let $\Lambda_\zeta$ be a component of $\Lambda_-$ or of $\Lambda_+$, and write $f(s,t) = e^t g_\zeta(s) + C_\zeta$ in local coordinates.  
	
	The key computation will show that $\theta^f \circ \phi - \theta = C - C_\zeta$ on the end defined by $\Lambda_\zeta$. We make use of the following identity from \cite[Proposition 9.18]{mcduff-salamon}:
	\begin{equation} \label{eq:isotopy-form}
		\phi_1^*(e^t \alpha) - e^t \alpha = d \int_0^1 (\iota_X e^t \alpha - f)\circ \phi_\tau\, d\tau.
	\end{equation}

	Denote function defined by the integral on the right by $\Theta$. By looking at primitives on both sides of Equation~\eqref{eq:isotopy-form}, we see that, for some constant $C$,
	\begin{equation} \label{eq:primitives}
		\theta^f \circ \phi_1 - \theta = \Theta + C.
	\end{equation}
	
	Along the end defined by $\Lambda_\zeta$, a direct computation of $X$ in local coordinates on $\bbR \times J^1 \Lambda_\zeta$ shows that $\iota_X e^t \alpha - f = -C_\zeta$, and hence, via Equation~\eqref{eq:primitives}, that 
	\begin{equation}
	\theta^f \circ \phi - \theta = C - C_\zeta.
	\end{equation}
	
	Since $\theta$ is constant along $\Lambda_-$ (resp.\ $\Lambda_+$), $\theta^f$ is constant along $\Lambda_-$ (resp.\ $\Lambda_+$) precisely when all of the constants $C_\zeta$ agree at $-\infty$ (resp.\ $+\infty)$, i.e.\ when $f$ is leveled.  This proves (1).
	
	To prove (2), notice that $L \cup L^f$ is leveled when, after adjusting $\theta^f$ by an overall constant, we have $\theta^f \circ \phi - \theta = 0$ at \emph{all} ends, which happens precisely when \emph{all} of the constants $C_\zeta$ are the same.
\end{proof}

\begin{remark}
An alternative, and more geometric, proof of Lemma~\ref{lem:disconnected_ends} involves building an exact Lagrangian strip connecting $L$ to $L^f$, such that the boundary of the strip is given by a path $\gamma\subset L$, a Legendrian path $\xi_+\subset \{t=T_+\} $, a path $\gamma^f \subset L^f$, and a Legendrian path $\xi_-\subset \{t=T_-\} $. That such a strip exists is a corollary of the techniques used in Section~\ref{ssec:geom-proof}, under the condition that $C_+=C_-$. We leave the details to the reader. Notice that, if $\Lambda_-$ or $\Lambda_+$ is disconnected, one needs to repeat the argument for each connected component.
\end{remark}

\subsection{Pushing off of $L$}
\label{ssec:push-off}

The goal of this section is to explain how to construct a globally cylindrical, non-singular function $f: L \to \R$, which, by Proposition~\ref{prop:displaced-cobordism}, yields an embedded leveled cobordism $L \cup L^f$. The condition that $f$ is non-singular implies that the gradient of such a function must have total winding number $\chi(L)$ along the boundary with respect to the incoming direction. In other words, the winding number of the gradient along $\Lambda_-$ is determined by adding $\chi(L)$ to the winding number along $\Lambda_+$.  The requirement that $f$ must be globally cylindrical is the main issue in the construction, and we will need to require that the winding number of the gradient along $\Lambda_+$ is at least $1$ to overcome the issue. 

The main construction in this section yields the following proposition.

\begin{proposition} \label{prop:construct-f}
Suppose $\Lambda_\pm$ are $r_\pm$-component Legendrian links in $(Y,\alpha)$ and $L$ is a Lagrangian cobordism from $\Lambda_-$ to $\Lambda_+$. There exists a non-singular globally cylindrical function $f:L\to \R$ so that
\begin{itemize}
\item $\Lambda_- \cup \Lambda_-^f$ is isotopic to the satellite of $\Lambda_-$ by 2-tangles with $1+2g_i+k_i$ full twists on the $i^{th}$ component for any choice of $g_i$ and $k_i$ such that $\sum g_i = g(L)$ and $\sum k_i = r_--1$, and 
    
\item $\Lambda_+ \cup \Lambda_+^f$ is isotopic to the satellite of $\Lambda_+$ by 2-tangles so that there is one full twist on a single component.
\end{itemize}
\end{proposition}

As a consequence, we obtain:

\begin{corollary} \label{cor:push-off-cobord}
 For any $n \geq 1$, the surface $\Sigma_n(L, f) := \bigcup_{i=0}^{n-1}\phi(\Gamma(\frac{i}{n}df))$  is a Lagrangian cobordism from the satellite of $\Lambda_-$ by the $r_-$-tuple of $n$-tangles $(\Delta^{1+2g_1+k_1}, \ldots, \Delta^{1+2g_{r_-}+k_{r_-}})$ where $\sum g_i = g(L)$ and $\sum k_i = r_--1$ to the satellite of $\Lambda_+$ by any $r_+$-tuple of $n$-tangles $(\Delta^{l_1}, \Delta^{l_2}, \ldots, \Delta^{l_{r_+}})$ where $\sum l_i = 1$. 
\end{corollary}

The intuition behind the construction for Proposition~\ref{prop:construct-f}, at least in the case where the Legendrian ends are connected, is to compactify $L_{[T_-,T_+]}$ by adding a disk $D_-$ at the bottom and a disk $D_+$ at the top.  The function $f$ is then chosen to have one source and one sink, both contained in $D_+$, and $2g$ saddles, all contained in $D_-$.  In particular, $f$ is non-singular on $L_{[T_-,T_+]}$.  It is not trivial, but true, that such a function $f$ can be required to be globally cylindrical.  That said, we will not directly pursue this intuitive approach. Instead, we provide an explicit construction of the function we need by breaking $L_{[T_-,T_+]}$ into planar pieces, defining a globally cylindrical function without critical points on each piece, and then gluing the domains and the functions together to construct the desired $f$.  

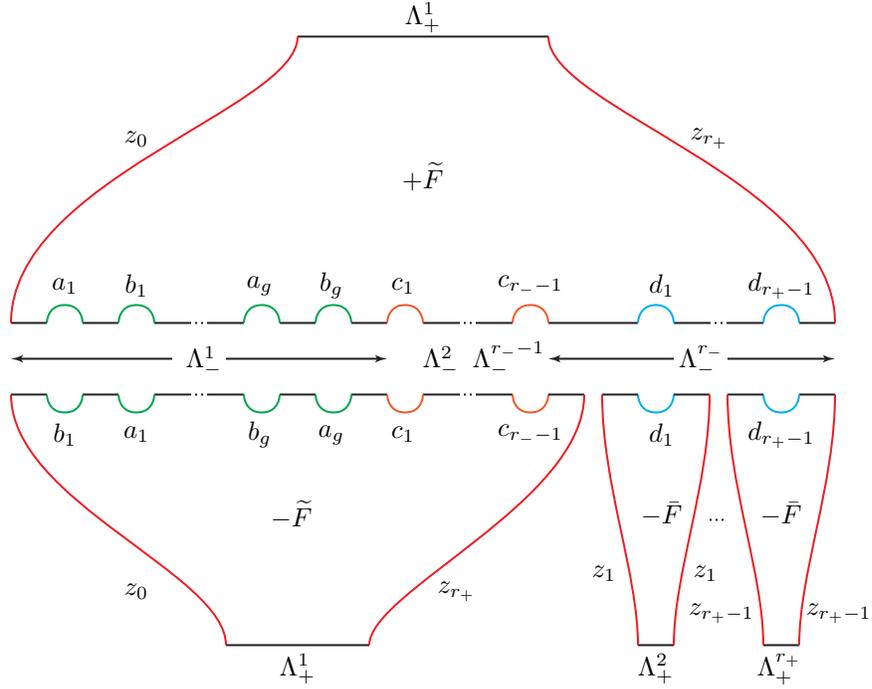
\begin{figure}
\labellist
    \small
    \pinlabel  $+\widetilde{F}$   at 216 245
    \pinlabel  $-\widetilde{F}$   at 150 75 
    \pinlabel  $-\bar{F}$   at  336 75
    \pinlabel  $-\bar{F}$   at  396 75
    \pinlabel  $a_1$   at 36 190
    \pinlabel  $a_1$   at 72 115
    \pinlabel  $b_1$   at 72 190
    \pinlabel  $b_1$   at 36 115
    \pinlabel  $a_g$   at 134 190
    \pinlabel  $a_g$   at 170 115
    \pinlabel  $b_g$   at 170 190
    \pinlabel  $b_g$   at 134 115
    \pinlabel  $c_1$   at 206 115
    \pinlabel  $c_1$   at 206 190
    \pinlabel  $c_{r_--1}$   at 270 115
    \pinlabel  $c_{r_--1}$   at 270 190
    \pinlabel  $d_1$   at 336 115
    \pinlabel  $d_1$   at 336 190
    \pinlabel  $d_{r_+-1}$   at 396 115
    \pinlabel  $d_{r_+-1}$   at 396 190
    \pinlabel  $z_0$   at 72 264
    \pinlabel  $z_0$   at 72 36
    \pinlabel  $z_{r_+}$   at 360 264
    \pinlabel  $z_{r_+}$   at 233 36
    \pinlabel  $z_1$   at 307 45
    \pinlabel  $z_1$   at 358 45
    \pinlabel  $z_{r_+-1}$   at 366 25
    \pinlabel  $z_{r_+-1}$   at 425 25
    \pinlabel  $\Lambda_-^1$   at 106 152
    \pinlabel  $\Lambda_-^2$   at 225 152
    \pinlabel  $\Lambda_-^{r_--1}$   at 259 153
    \pinlabel  $\Lambda_-^{r_-}$   at 356 152
    \pinlabel  $\Lambda_+^1$   at 153 -3
    \pinlabel  $\Lambda_+^{2}$   at 333 -3
    \pinlabel  $\Lambda_+^{r_+}$   at 395 -3
    \pinlabel  $\Lambda_+^{1}$   at 216 325
    \endlabellist
\centerline{\includegraphics[width=4.5in]{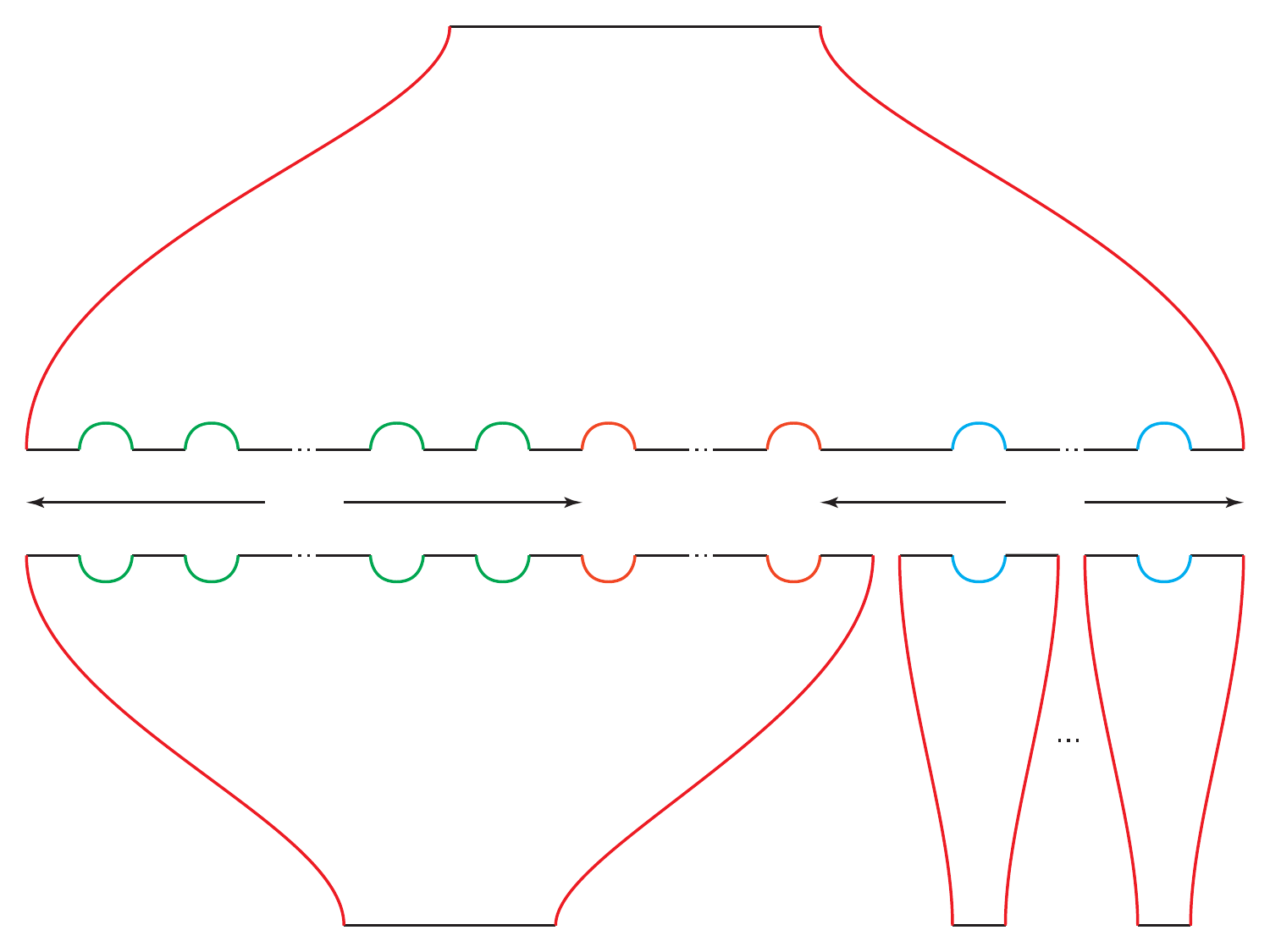}}
\caption{Gluing $r_+ +1$ combs $C_n$ according to the pattern depicted in the figure yields a genus $g$ surface $\Sigma$ with $r_- + r_+$ boundary components. The (unglued) intervals $I_k$ are horizontal segments, while the (glued) intervals $J_k$ are curved segments.}
\label{fig:comb-glue}
\end{figure}

We begin the construction by specifying how to build $L_{[T_-,T_+]}$ --- which we think of as an abstract surface $\Sigma$ with $r_+$ ``positive'' boundary components, $r_-$ ``negative'' boundary components, and genus $g$ --- out of planar pieces. Each planar piece will be an \allbf{$n$-comb} $C_n$, i.e.\ a disk with $n+2$ distinguished and disjoint closed intervals $I_0, \ldots, I_{n+1}$ along its boundary; denote the intervals in the closure of the complement of $\bigcup I_k$ by $J_0, \ldots, J_{n+1}$. To assemble $\Sigma$ out of $r_+ + 1$ $n$-combs, we glue the $n$-combs along the intervals $J_i$ according to the identifications in Figure~\ref{fig:comb-glue}.  Figure~\ref{fig:comb-decomp} depicts the result $\Sigma$ of the gluing process in Figure~\ref{fig:comb-glue}, demonstrating how the curves labeled $a_i$ and $b_i$ come together to create genus, how the curves $c_i$ yield additional ``negative'' boundary components, and how the curves $d_i$ and $z_i$ (for $1 \geq i \leq r_+-1$) produce additional ``positive'' boundary components. It is straightforward to check by computing the Euler characteristic and counting connected components of the boundary that the gluing pattern yields the desired surface $\Sigma$ with the correct number of boundary components and genus.  Note that the precise recipe for decomposing $\Sigma$ into combs matters less than the fact that $\Sigma$ \emph{can} be decomposed into combs.

\begin{figure}
\labellist
\small\hair 2pt
 \pinlabel {$\Lambda^1_+$} [l] at 115 156
 \pinlabel {$\Lambda^2_+$} [l] at 313 156
 \pinlabel {$\Lambda^1_-$} [r] at 28 23
 \pinlabel {$\Lambda^2_-$} [t] at 204 13
 \pinlabel {$\Lambda^3_-$} [l] at 313 23
 \pinlabel {$z_0$} [r] at 25 85
 \pinlabel {$z_1$} [l] at 317 88
 \pinlabel {$z_2$} [tl] at 193 66
 \pinlabel {$a_1$} [l] at 38 62
 \pinlabel {$b_1$} [l] at 59 62
 \pinlabel {$a_2$} [l] at 103 50
 \pinlabel {$b_2$} [l] at 122 50
 \pinlabel {$c_1$} [b] at 171 42
 \pinlabel {$c_2$} [b] at 233 42
 \pinlabel {$d_1$} [bl] at 265 101
\endlabellist
\centerline{\includegraphics{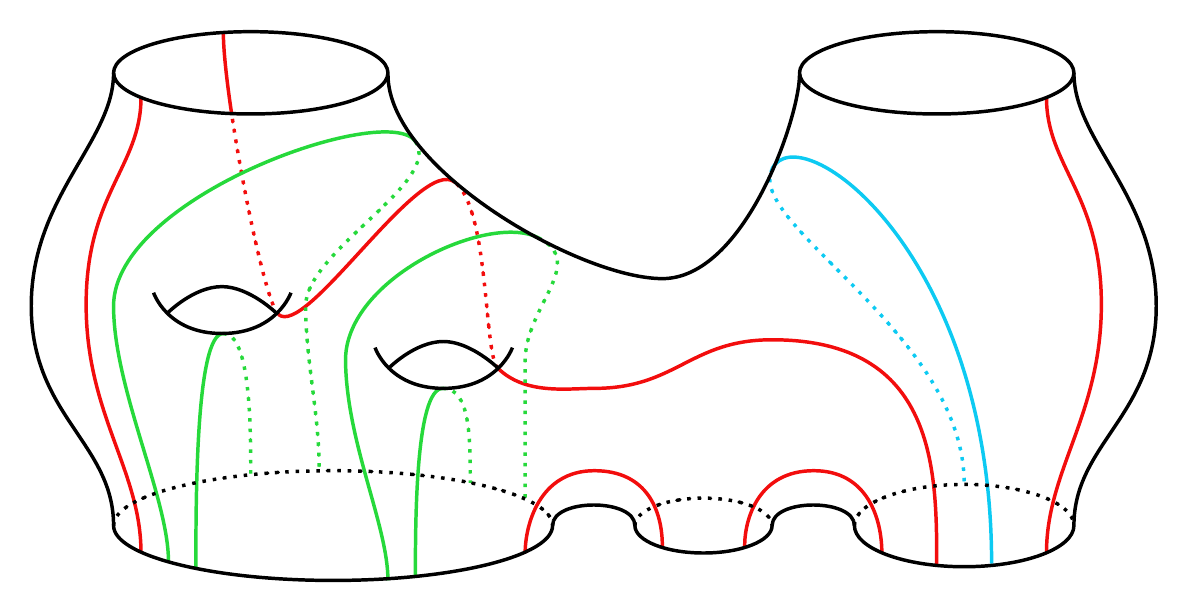}}
\caption{The gluing instructions on the combs in Figure~\ref{fig:comb-glue} yield a surface with $r_+$ ``positive'' boundary components (here, $r_+=2$), $r_-$ ``negative'' boundary components (here, $r_-=3$), and genus $g$ (here, $g=2$).}
\label{fig:comb-decomp}
\end{figure}
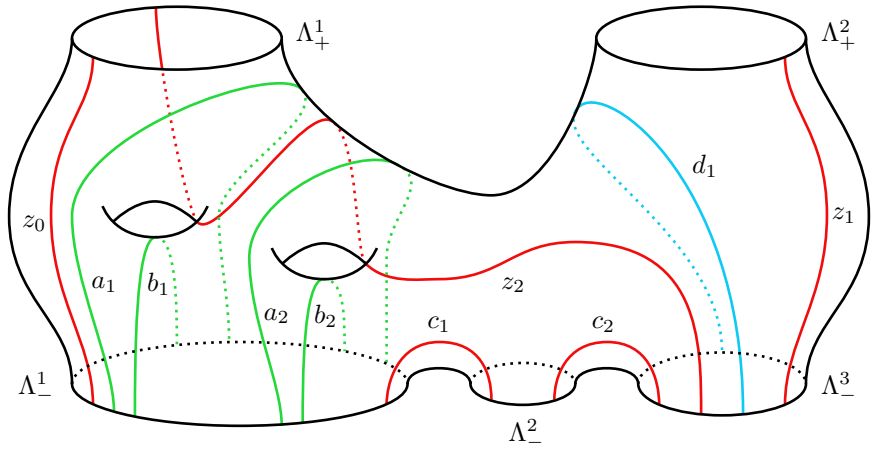

\begin{remark} \label{rem:redistribute}
The gluing pattern depicted in Figure~\ref{fig:comb-glue} corresponds to putting $2g(L)+1$ twists on the first component of $\Lambda_-$, $r_+$ twists on the last component, and one twist on each of the other components. By shifting the endpoints of the curves labeled $a_i$, $b_i$, and $d_i$ to other components of $\Lambda_-$ in Figure~\ref{fig:comb-decomp}, we can redistribute the endpoints of these intervals onto different negative boundary components of $\Sigma$.  In the gluing picture in Figure~\ref{fig:comb-glue}, the redistribution is tantamount to interspersing the intervals labeled $a_i$, $b_i$, and $d_i$ between the intervals labeled $c_i$.
\end{remark}

The next step in the construction is to define non-singular functions $\widetilde{F}_n, \bar{F}_n: C_n \to \R$ that are globally cylindrical.  Note that while we will write down $C^1$ functions that are piecewise $C^\infty$, we will assume that such functions have been smoothed to $C^\infty$ functions through modification in arbitrarily small neighborhoods of the non-smooth points; if the $C^1$ functions have no critical points, then neither will the smoothed functions.

The functions $\widetilde{F}_n$ and $\bar{F}_n$ will arise from interpolating between pairs of functions at the ends of $C_n$.  Define functions $\tilde{g}^\pm_n, \bar{g}^\pm_n: [-\pi/2,3\pi/2] \to \bbR$ by smoothing the following piecewise $C^1$ functions.  See Figure~\ref{fig:end-functions}.
\begin{align*}
    \tilde{g}^+_n (s) &= \begin{cases} 1 & s \in [0,\pi] \\ \cos 2s & \text{otherwise} \end{cases} &  \bar{g}^+_n (s) & = 1 \\
    \tilde{g}^-_n (s) &= \begin{cases} 1 & s \in [0,\pi/2] \\ \cos 4ns & s \in [\pi/2,\pi] \\
    \cos 2s & \text{otherwise} \end{cases} & \bar{g}^-_n(s) &= \begin{cases} \cos 4ns & s \in [\pi/2,\pi] \\ 1 & \text{otherwise} \end{cases}
\end{align*}

\begin{figure}
    \labellist
    \small
    \pinlabel $\tilde{g}^+$ [l] at 144 306
    \pinlabel $\tilde{g}_n^-$ [l] at 144 126
    \pinlabel $\bar{g}^+$ [l] at 396 306
    \pinlabel $\bar{g}_n^-$ [l] at 396 126
    \pinlabel $n$   at 108 6
    \pinlabel $n$   at 360 6
    \pinlabel $-\frac{\pi}{2}$ [r] at 16 74
    \pinlabel $\frac{3\pi}{2}$ [l] at 160 74
    \pinlabel $1$ [r] at 16 150
    \pinlabel $-1$ [r] at 16 24
    \pinlabel $-\frac{\pi}{2}$ [r] at 16 251
    \pinlabel $\frac{3\pi}{2}$ [l] at 160 251
    \pinlabel $1$ [r] at 16 327
    \pinlabel $-1$ [r] at 16 201
    \pinlabel $-\frac{\pi}{2}$ [r] at 269 74
    \pinlabel $\frac{3\pi}{2}$ [l] at 413 74
    \pinlabel $1$ [r] at 269 150
    \pinlabel $-1$ [r] at 269 24
    \pinlabel $-\frac{\pi}{2}$ [r] at 269 251
    \pinlabel $\frac{3\pi}{2}$ [l] at 413 251
    \pinlabel $1$ [r] at 269 327
    \pinlabel $-1$ [r] at 269 201
    \endlabellist
    \centering
    \includegraphics[width=4.5in]{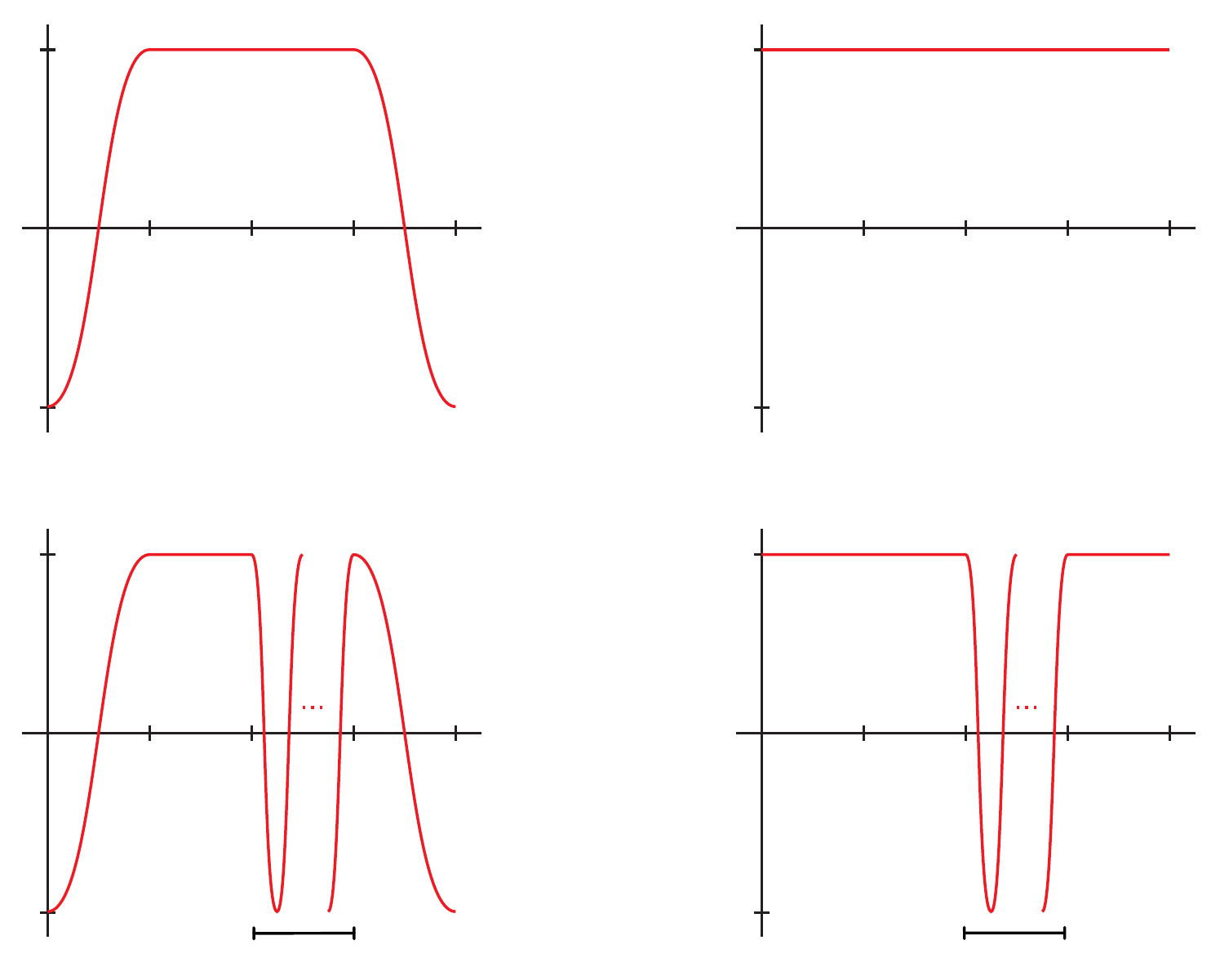}
    \caption{The twisted functions $\tilde{g}^\pm_n$ (left) and flat functions $\bar{g}^\pm_n$ (right) at the ends of the functions $\widetilde{F}, \bar{F}$.}
    \label{fig:end-functions}
\end{figure}
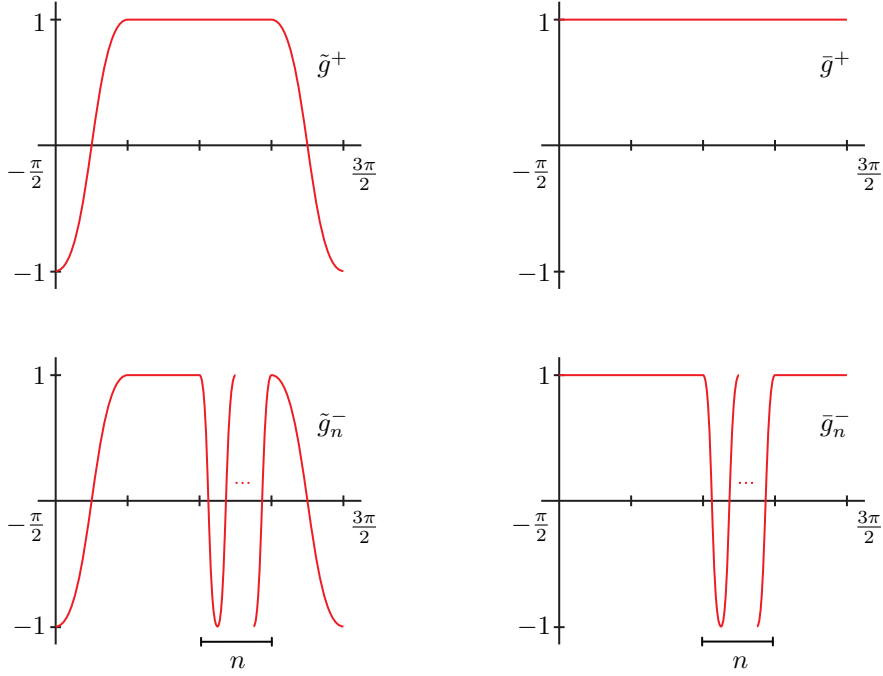

We next define a function $\widetilde{F}_n: [-\pi/2,3\pi/2] \times [T_-, T_+] \to \R$ by interpolating between $\tilde{g}^\pm_n$.  For some fixed $\epsilon < \frac{1}{2} (T_+-T_-)$, let $\sigma: \R \to \R$ be a smooth, increasing function that is $0$ on $(-\infty, T_-+\epsilon)$ and $1$ on $(T_+-\epsilon,\infty)$. The function $F$ is then given by
\[\widetilde{F}(s,t) = e^t\bigl[(1-\sigma(t))\tilde{g}_-(s) + \sigma(t) \tilde{g}_+(s) \bigr]; \]
The function $\bar{F}_n$ is defined similarly using $\bar{g}^\pm_n$. See Figure~\ref{fig:interpolating-function}.

\begin{figure}
    \centerline{
    \includegraphics[width=3in]{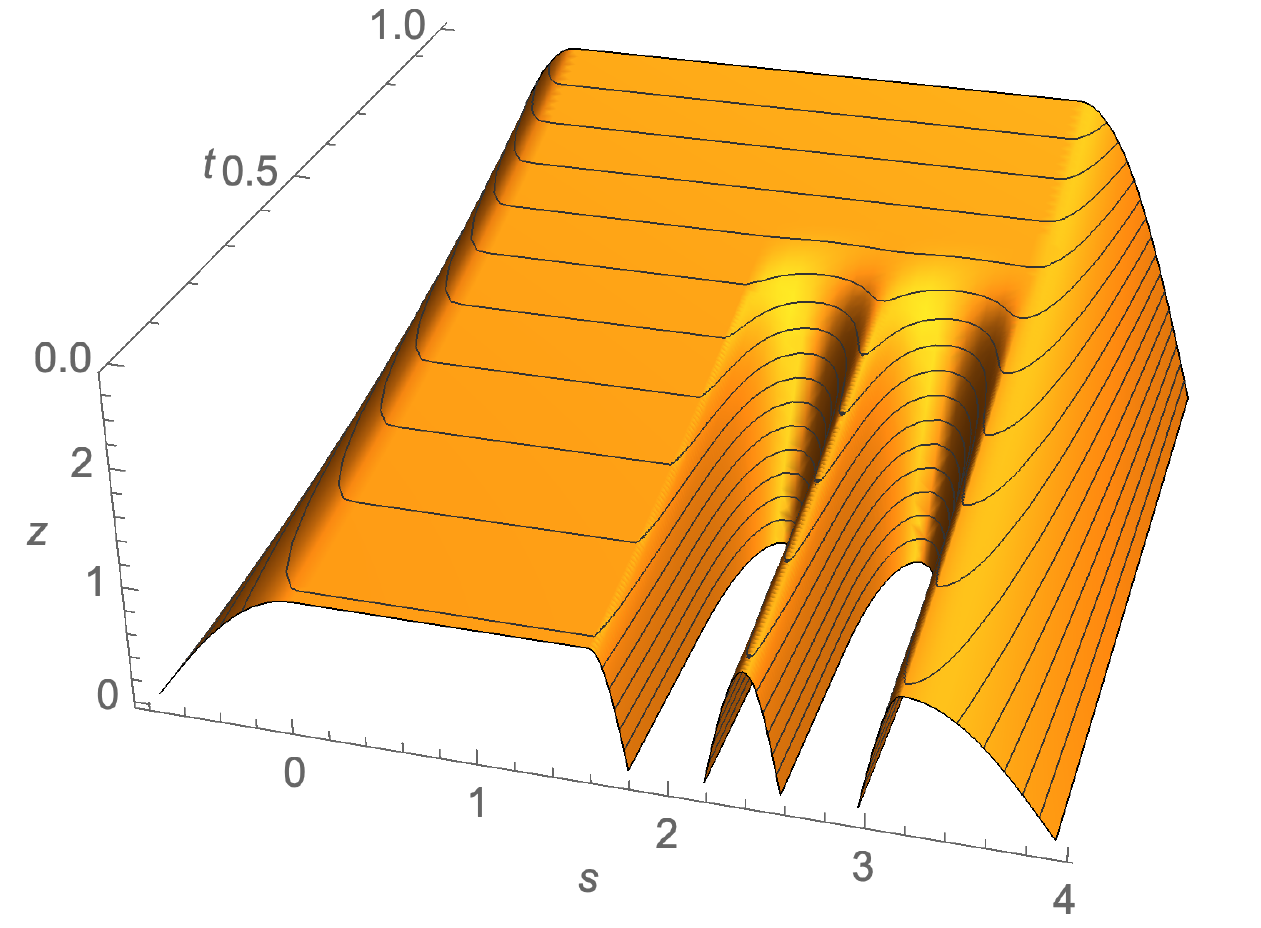} }
    \caption{The function $\widetilde{F}_2$ that interpolates between $\tilde{g}^-_2$ and $\tilde{g}^+_2$ on the comb $C_2$.  Only the portions of the plot with $\widetilde{F}_2\geq 0$ are shown.}
    \label{fig:interpolating-function}
\end{figure}

The next step is to carve an $n$-comb $C_n$ out of the domain of $\widetilde{F}_n$ by taking the points on which $\widetilde{F}(s,t) \geq 0$ and letting the intervals $I_k$ be the intersection of this set with the top and bottom edges of the domain, numbered so that $I_0$ is $[-\pi/2, 3\pi/2] \times \{T_+\}$.  We again denote by $\widetilde{F}_n$ its restriction to $C_n$.  Apply a similar procedure to $\bar{F}_n$.

\begin{lemma} \label{lem:no-crit-points}
    The function $\widetilde{F}_n: C_n \to \bbR$ (resp. $\bar{F}_n$) is non-singular and cylindrical over $\tilde{g}^\pm_n$ (resp. $\bar{g}^\pm_n$) at the ends with respect to the given coordinates on $C_n$.
\end{lemma}

The lemma follows from a straightforward computation to find the common zeros of the partial derivatives of $\widetilde{F}$, and then recognizing that those common zeros must have negative $\widetilde{F}_n$ value.

The penultimate step in the construction is to assign functions $\pm \widetilde{F}$ and $\pm \bar{F}$ to the combs in Figure~\ref{fig:comb-glue} according to the markings in the figure.  Define $f: \Sigma \to \bbR$ piecewise using the assignments in Figure~\ref{fig:comb-glue}.  Note that $f$ is, indeed, $C^1$ and hence can be smoothed since in collar coordinates along the intervals $J_i$, $f$ is odd with respect to the transverse coordinate.

To complete the construction, we transport the function $f$ to $L$. First,  specify a diffeomorphism $\phi: \Sigma \to L_{[T_-,T_+]}$ with the property that the coordinates $t$ near $T_\pm$ at the ends of the combs that make up $\Sigma$ corresponds to the coordinate $t$ near the ends of $L$. Second, extend $f \circ \phi^{-1}$ to all of $L$.

Finally, we check that the function $f: L \to \bbR$ does, indeed, satisfy the requirements of Proposition~\ref{prop:construct-f}.  The fact that $f$ is non-singular and globally cylindrical comes directly from the construction; in particular, $f$ is leveled since there are embedded strips along which $f$ is exponential from the negative end to \emph{each} positive boundary component since each such component comes from its own comb.  The assertions about the twisting at the ends of the cobordism follow from counting the number of zeros of the functions $g^\pm_n$ used in the construction together with Remark~\ref{rem:redistribute}.

\subsection{Inserting the Tangle Cobordism}
\label{ssec:geom-proof}

The final step of the proof of Proposition~\ref{prop:main_geom}, and hence of Theorem~\ref{thm:geom}, is to insert the tangle cobordism $\Pi_-\prec_P\Pi_+$ into the cobordism $\Sigma_n(L,f)$ built in Corollary~\ref{cor:push-off-cobord}.  

\begin{proof}[Proof of Proposition~\ref{prop:main_geom}]
Given a Lagrangian cobordism $\Lambda_- \prec_L \Lambda_+$, Lemma~\ref{tubnbhd} yields an exact symplectic embedding $\phi: U \subset T^*L \to \bbR \times Y$ with cylindrical ends.  Construct a globally cylindrical function $f$ as in Proposition~\ref{prop:construct-f}. Let $\Phi \simeq \gamma \times (-3\delta, 3\delta)$ be the embedded strip along which $f$ is cylindrical, as per Definition \ref{defn:glob-cyl}.  Define two nested neighborhoods $N_k \subset N$, $k=1,2$, to be the portions of $\Phi$ coming from $\gamma \times (-k\delta,k\delta)$. 

We now insert the tangle cobordism $P$ into $T^*N_1$.  In preparation, shrink $P$ so that it lies in $\bbR \times B(\delta) \subset \bbR \times J^1(-\delta,\delta)$; this can be achieved, for example, by a uniform scaling in all coordinates.  We assume $\delta$ is sufficiently small so that the neighborhood embedding $\phi:T^*L\to \R\times Y$ is defined for all vectors of norm at most $2\delta$. Using the identification map $\beta$ described in Section~\ref{ssec:tubular}, we may think of $P$ as lying in $T^*N_2$; in particular, notice that, in $T^*N_2 \setminus T^*N_1$, $P$ coincides with the graphs of the forms $d(\frac{i}{n} e^t)$ for $i \in \{0, \ldots, n-1\}$. Thus, since $f=e^t+C$ inside $N_2$, we may form the exact Lagrangian $\left(\Sigma_n(L,f) \setminus T^*N_1 \right) \cup \left( P \cap T^*N_2\right)$, which we may then embed into $\bbR \times Y$ via $\phi$ to form $\Sigma(L,P)$.

We need to check that $\Sigma(L,P)$ is still an exact Lagrangian.  Since all of the symplectomorphisms we use are exact (as checked in Section \ref{ssec:tubular}), we can work in $T^*N_1 \subset T^*L$.  Since $P$ is leveled, the primitive $\theta_P$ for the symplectic form along $P$ is constant along the bottom of $P$ (where we might as well take the constant to be $0$) and along the top of $P$.  Since $P$ is cylindrical along the boundary, the primitive $\theta_P$ is constant there as well.  Thus, the primitive $\theta_P$ vanishes near its boundary, and hence matches up with both $\theta$ and $\theta^f$ there since $f$ is equal to $e^t$ in $N_1$.  

Finally, we know that $\Sigma(L,P)$ is leveled since $f$ it is globally cylindrical by construction.  Finally, the cobordism is orientable as long as all the data ($L$,$P$,$\lambda_\pm$,$\Pi_\pm$, $\Sigma(\Lambda_+,\Pi_+)$,$\Sigma(\Lambda_-,\Pi_-)$) are.
\end{proof}

The proof of Proposition~\ref{prop:main_geom} yields the tools necessary to prove Proposition~\ref{main_geom_nonorientable}.

\begin{proof}[Proof of Proposition \ref{main_geom_nonorientable}]
The first step in the proof is to add a negative boundary component to the original cobordism $L$ so that we may move the twists coming from the proof of Proposition~\ref{prop:main_geom} off of $\Lambda_-$.  To accomplish this, let $L_\wedge$ be a $\wedge$-handle cobordism from $\Lambda_- \cup \Upsilon$, where $\Upsilon$ is the maximal Legendrian unknot, to $\Lambda_-$.  We then pre-compose $L$ with $L_\wedge$ to yield a new cobordism $L_\wedge \odot L$ from $\Lambda_- \cup \Upsilon$ to $\Lambda_+$.

We can now apply Proposition \ref{prop:main_geom} to the cobordisms $L_{\wedge}\odot L$ and $P$, making sure to place  $2g(L)+1$ twists on the $\Upsilon$ component. The result is a cobordism \[\left(\Sigma(\Lambda_-,\Delta \Pi_-)\cup\Sigma(\Upsilon,\Delta^{2g(L)+1})\right)\prec\Sigma(\Lambda_+,\Delta \Pi_+).\] As long as we do not care about orientations, we know that $\Sigma(\Upsilon,\Delta^n)$ is Lagrangian fillable for any $n\geq 1$ (see Lemma \ref{lem:lifted-0h}). Adding this filling to the bottom of the cobordism yields the desired cobordism  
$\Sigma(\Lambda_-,\twist \Pi_-)\prec\Sigma(\Lambda_+,\twist \Pi_+)$.
\end{proof}

Note that an orientable Lagrangian filling of $\Sigma(\Upsilon, \twist^n)$ induces a uniform orientation on the boundary. This orientation agrees with the orientation on $L_\wedge \odot L$ only when $\Pi_-$  has a uniform orientation.

\section{Satellite construction for decomposable cobordisms}
\label{decomp}

The goal of this section is to construct \emph{decomposable} cobordisms between satellites. The idea of the construction is to start with a decomposable cobordism $\mathbf{L} = E_1 \odot \cdots \odot E_m$, where the $E_k$ are elementary cobordisms with $\Lambda_{k-1} \prec_{E_k} \Lambda_k$. At each level of the cobordism, we take the $n$-copy of the link $\Lambda_k$. We want to upgrade the 0- and 1- handles of $\mathbf{L}$ to diagrammatic moves between the new $n$-stranded diagrams.  The key is to use the isotopy depicted in Figure~\ref{fig:half-at-cusp} to enable the attachment of  $0$- and $1$- handles, and hence to obtain the upgraded moves depicted in Figures~\ref{fig:lifted-0h}, \ref{fig:lifted-1h-half}, and \ref{fig:lifted-1h-full}.

These diagrammatic moves require extra twists.  To keep track of the twists, we introduce a \textbf{twist function} (see Definition \ref{def:twist-fn}), which counts the number of twists we need at each level $k$ of the cobordism $\mathbf{L}$ for each connected component of the link $\Lambda_k$. We then have two problems to address:
\begin{enumerate}
    \item Given a twist function, what more do we need to be able to build a cobordism between satellites? This is answered in Theorem \ref{thm:decompTwist}.
    \item Do twist functions with the specific properties needed in Theorem \ref{thm:decompTwist} always exist? This is answered in Proposition \ref{prop:exists-tw-fn}.
\end{enumerate}
Finally, in Corollary \ref{cor:decompGen}, we list many situations in which we do indeed get decomposable cobordisms between satellites.  

\subsection{Satellites of elementary cobordisms}

We begin by building $n$-stranded analogs of 0-handles and 1-handles, which we will refer to as $\Sigma$-0-handles and $\Sigma$-1-handles. The isotopy shown in Figure~\ref{fig:half-at-cusp}, which is used in both lemmas, consists of $n$ Reidemeister II moves.

\begin{figure}
\begin{gather*}
\twist^{1/2}_s\hspace{16.67em}\\ 
\begin{gathered}\includegraphics[scale=\figscale]{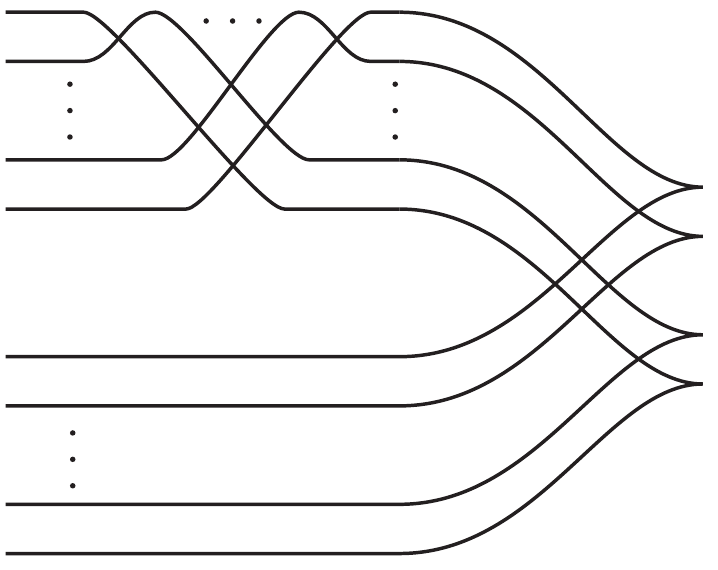}\end{gathered}
\hspace{0.67em}\simeq\hspace{0.67em}
\begin{gathered}\includegraphics[scale=\figscale]{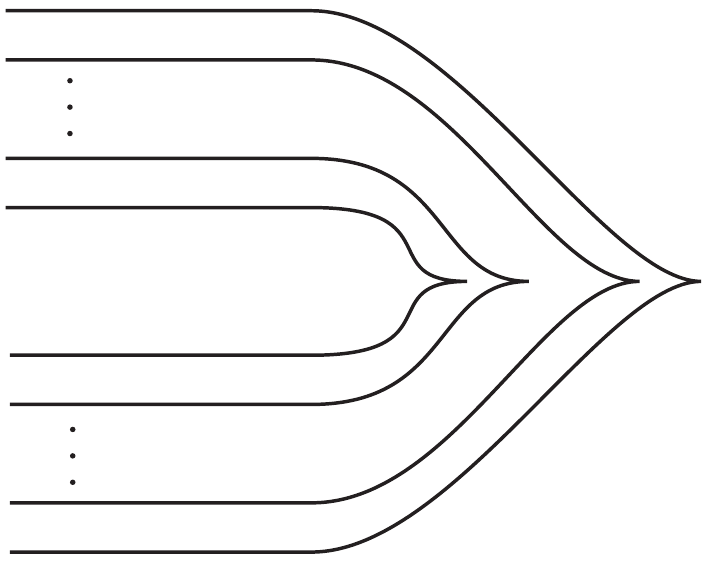}\end{gathered}
\end{gather*}
\caption{A key isotopy.}
\label{fig:half-at-cusp}
\end{figure}

\begin{lemma}[$\Sigma$-0-handle]\label{lem:lifted-0h}
For any $k \in \frac{1}{2}\bbN$, there exists an orientable decomposable cobordism $\emptyset \prec \lsat{}{\Upsilon}{\twist^{1+k}_s}$ if $k = 0$ or the orientation $s$ is uniform.
\end{lemma}

\begin{figure}
\labellist
    \small
    \pinlabel $\twist^k_s$ [l] at 216  135 
    \pinlabel $\twist^k_s$ [l] at 216 -100 
\endlabellist
\centering
\includegraphics[scale=\figscale]{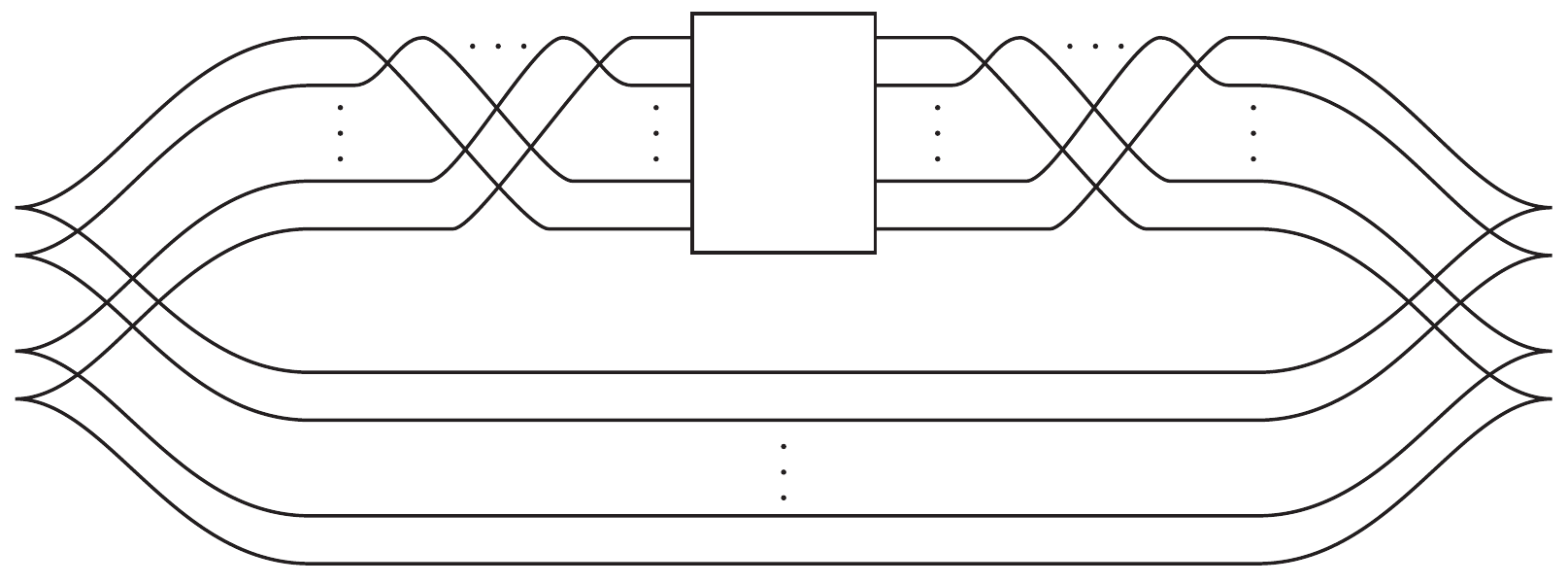} \\
\hspace{1pt}\rotatebox{-90}{$\simeq$}\vspace{12pt} \\
\includegraphics[scale=\figscale]{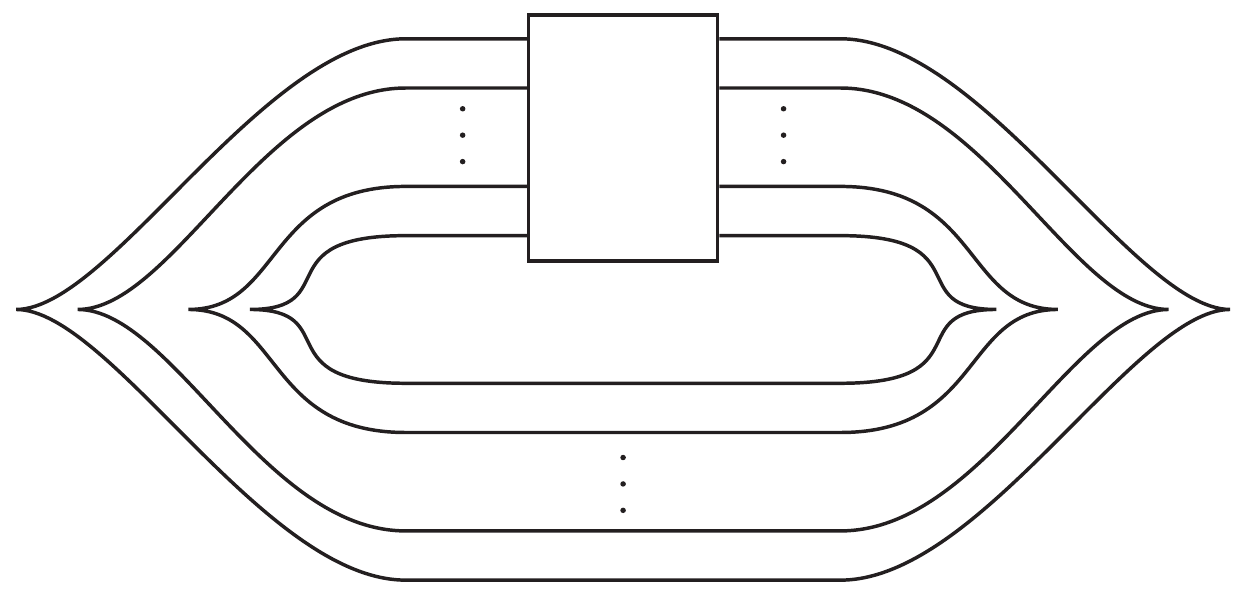}
\caption{Applying the isotopy in Figure~\ref{fig:half-at-cusp} to both sides of the canonical diagram of $\lsat{}{\Upsilon}{\twist^{1+k}_s}$.}
\label{fig:lifted-0h}
\end{figure}

\begin{proof}
Applying the isotopy in Figure~\ref{fig:half-at-cusp} to both sides of the  diagram of $\lsat{}{\Upsilon}{\twist^{1+k}_s}$ gives the isotopy shown in Figure~\ref{fig:lifted-0h}. When $k = 0$, the pattern $\twist^k_s$ is trivial, and so $\lsat{}{\Upsilon}{\twist^{1+k}_s}$ is isotopic to $n$ disjoint copies of the unknot. The desired cobordism $\emptyset \prec \lsat{}{\Upsilon}{\twist^{1+k}_s}$ is therefore comprised of $n$ 0-handles followed by the isotopy described above. 

When $k > 0$, if the orientation $s$ is uniform then $\twist^k_s$ contains only positive crossings. By \cite{positivity}, this means a decomposable cobordism from $\emptyset$ to this link exists, and so the desired cobordism $\emptyset \prec \lsat{}{\Upsilon}{\twist^{1+k}_s}$ is comprised of the aforementioned cobordism followed by the isotopy described above.
\end{proof}

\begin{lemma}[$\Sigma$-1-handle]\label{lem:lifted-1h-half}
There exists an orientable decomposable cobordism from the lower diagram of Figure~\ref{fig:lifted-1h-half} to the upper diagram of Figure~\ref{fig:lifted-1h-half} if the orientations on the left and right side  agree.
\end{lemma}

\begin{figure}
\labellist
    \small
    \pinlabel $s$             [l] at 452 -316 
    \pinlabel $s$             [r] at  -5 -316 
    \pinlabel $-\overline{s}$ [l] at 452 -414 
    \pinlabel $-\overline{s}$ [r] at  -5 -414 
    \pinlabel $s$             [l] at 452 140 
    \pinlabel $s$             [r] at  -5 140 
    \pinlabel $-\overline{s}$ [l] at 452  42 
    \pinlabel $-\overline{s}$ [r] at  -5  42 
\endlabellist
\centering
\includegraphics[scale=\figscale]{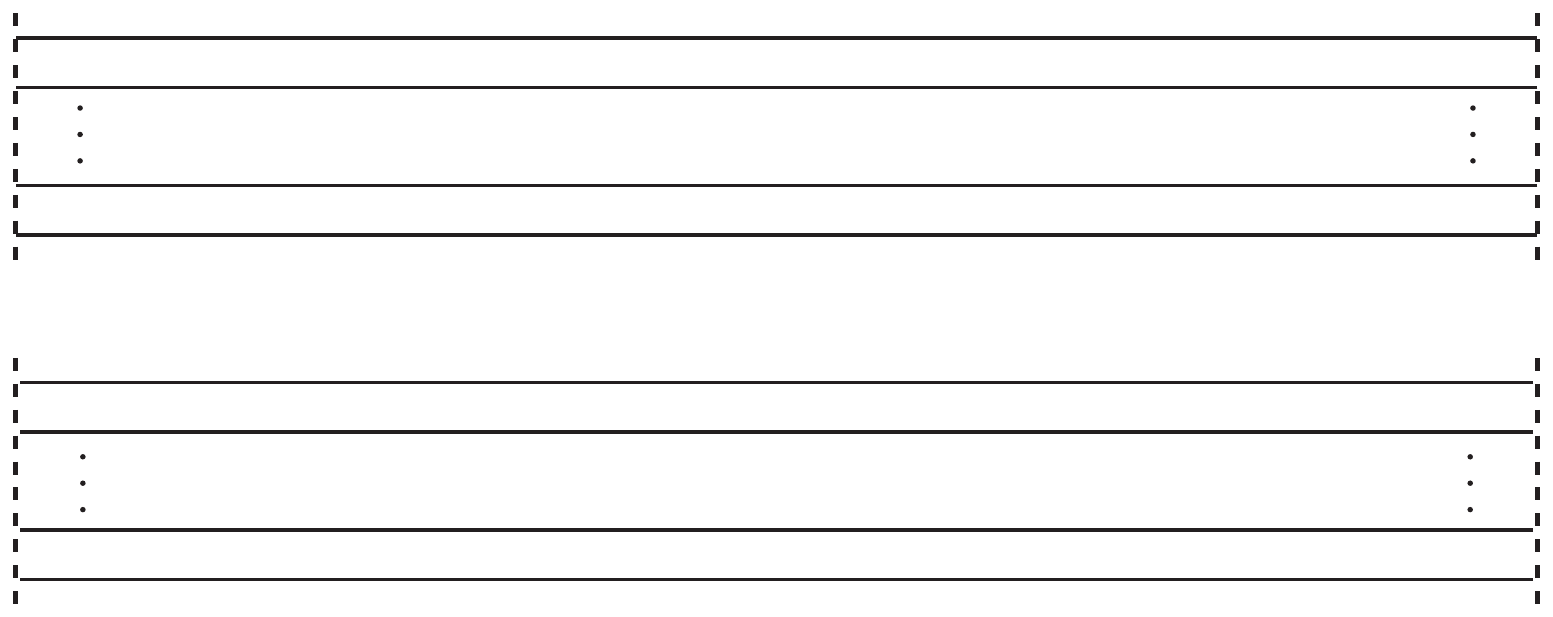} \\
\vspace{6pt}\hspace{1pt}$\uparrow$\vspace{3pt} \\
\includegraphics[scale=\figscale]{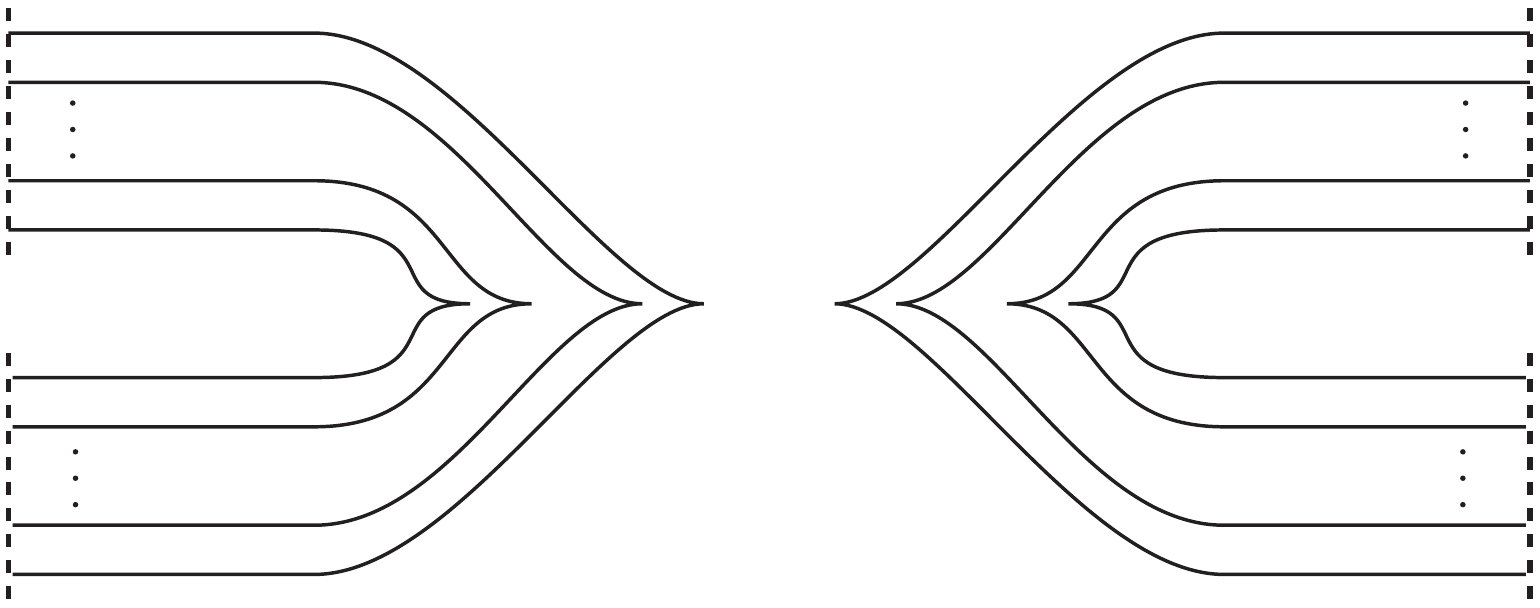} \\
\vspace{6pt}\hspace{1pt}\rotatebox{-90}{$\simeq$}\vspace{3pt} \\
\includegraphics[scale=\figscale]{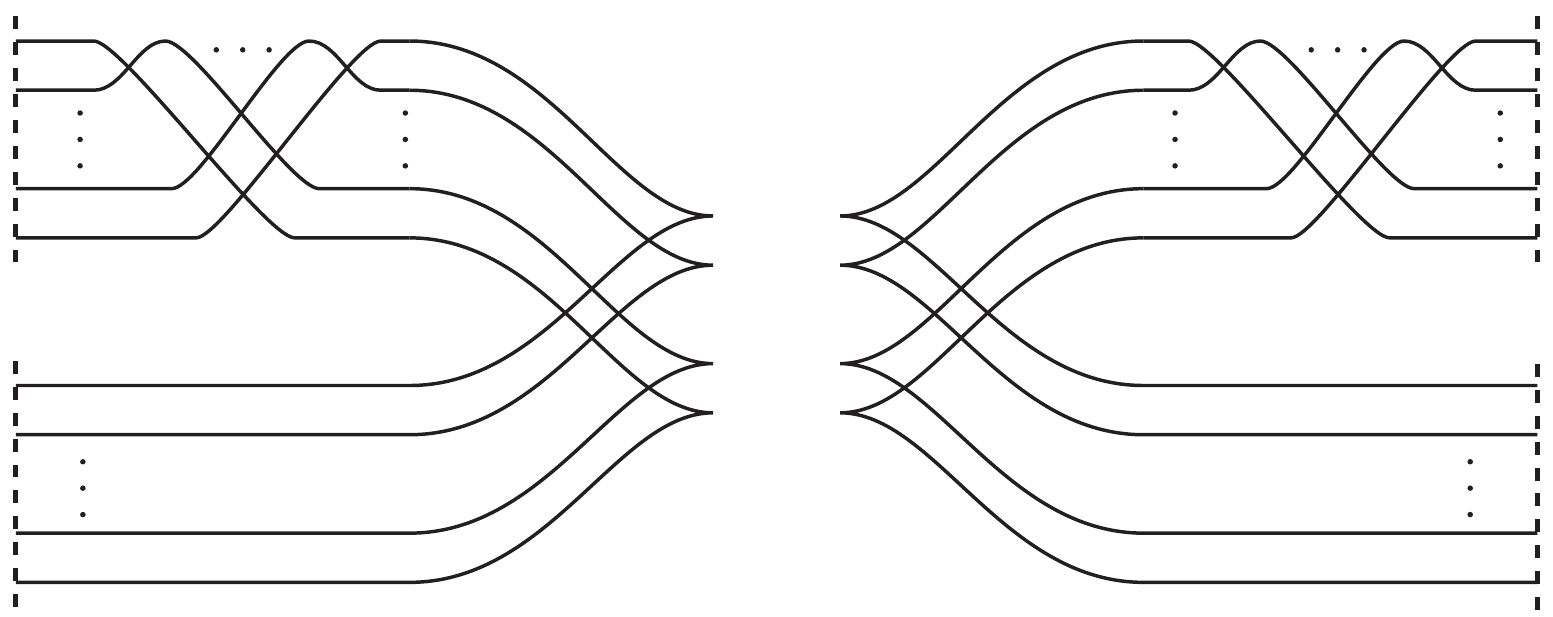}
\caption{The cobordism of Lemma~\ref{lem:lifted-1h-half}, where the induced orientations on the upper and lower $n$-strands at the boundaries must agree as specified.}
\label{fig:lifted-1h-half}
\end{figure}
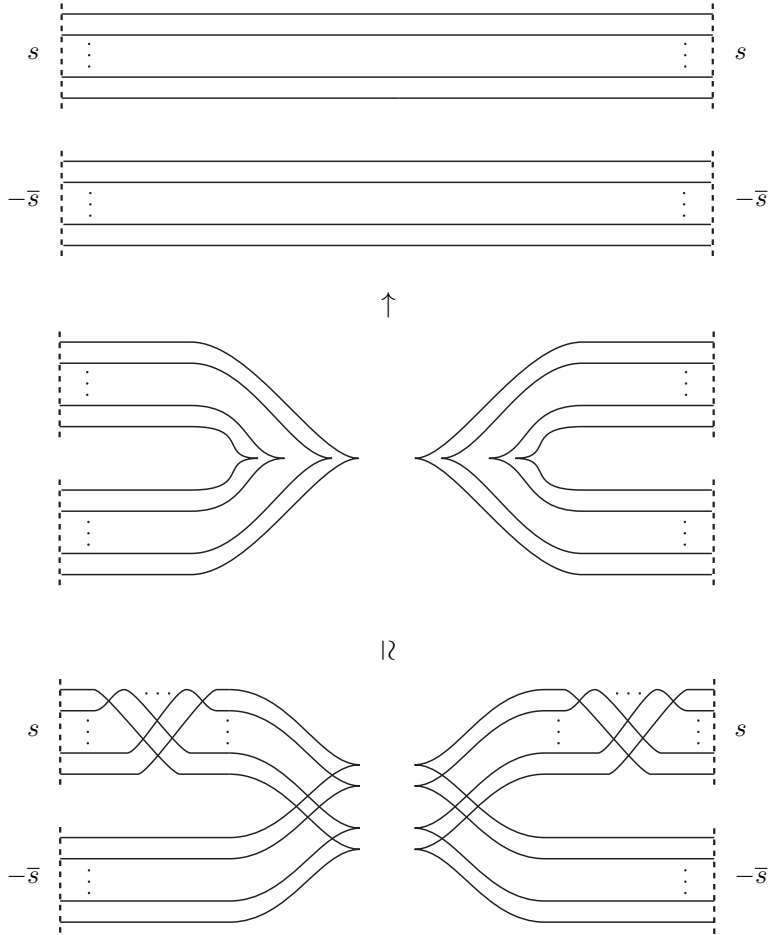

\begin{proof}
Apply the isotopy in Figure~\ref{fig:half-at-cusp} to both sides of the lower diagram of Figure~\ref{fig:lifted-1h-half}. Then, successively apply $n$ 1-handles to the middle diagram to obtain the upper diagram of Figure~\ref{fig:lifted-1h-half}.
\end{proof}

Note that without the presence of more twists, the only time the lower and upper diagrams of Figure~\ref{fig:lifted-1h-half} can appear in the diagram of a satellite is if the orientation $s$ on the pattern is symmetric, and therefore $s = \overline{s}$. This is illustrated by Figure~\ref{fig:cusp-ex}.
To handle non-symmetric orientations $s$, we introduce more twists to obtain the full $\Sigma$-1-handle depicted in Figure \ref{fig:lifted-1h-full}.

\begin{figure}
\labellist
    \small
    \pinlabel $s$             [r] at -2 361 
    \pinlabel $-s$            [r] at -2 263 
    \pinlabel $s$             [r] at 226 361 
    \pinlabel $-\overline{s}$ [r] at 226 263 
    \pinlabel $s$             [r] at 74 138 
    \pinlabel $-s$            [r] at 74 40 
\endlabellist
\centering
\includegraphics[scale=\figscale]{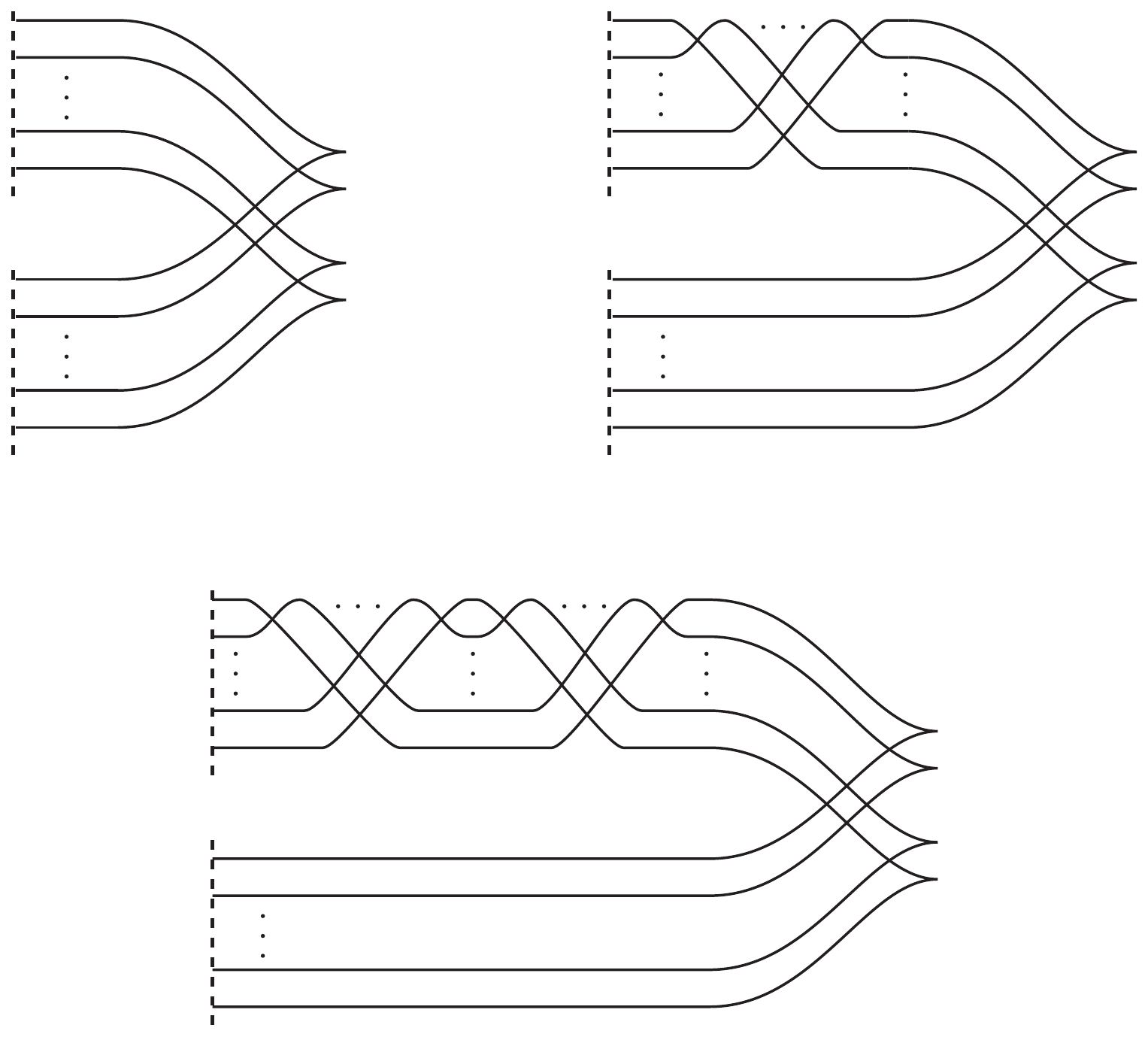}
\caption{Induced orientations at an $n$-stranded cusp with different numbers of twists. In the $\Delta^{1/2}$ case, the lower $n$ strands have the opposite orientation: $-\overline{s}$.}
\label{fig:cusp-ex}
\end{figure}
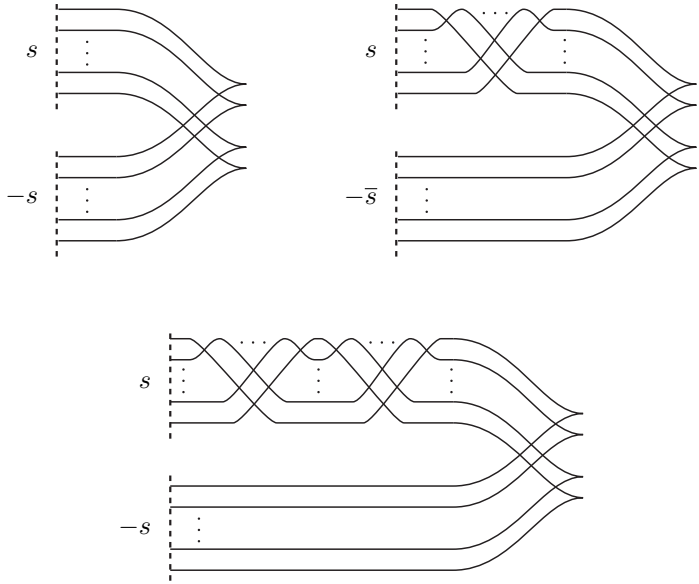

\subsection{Twist functions}

Given a decomposable cobordism $\Lambda_- \prec_{\mathbf{L}} \Lambda_+$, with $\mathbf{L} = E_1 \odot \cdots \odot E_k$ where each $E_k$ is an elementary cobordism $\Lambda_{k-1} \prec_{E_k} \Lambda_k$, let $\components(\Lambda_k)$ be the set of components of the link $\Lambda_k$, and $\components(\mathbf{L}) = \bigcup_k \components(\Lambda_k)$. Define
\[ O(\mathbf{L}) = \{ \Lambda \in \components(\mathbf{L}) : \Lambda \text{ is added by a 0-handle} \}. \] 

In order to capture when $\Sigma$-handles can be applied, we introduce the notion of a twist function $f : \components(\mathbf{L}) \to \halfN$, which counts how many twists (or half twists) we wish to add to each component. 

When it is clear from the context, we use the notation $\Lambda \prec_{E_k} \Lambda'$, where $\Lambda \subset \Lambda_{k-1}$ and $\Lambda' \subset \Lambda_k$, to indicate a connected component of the elementary cobordism $E_k$. Notice that if $E_k$ is a 0- or 1-handle, then for all but exactly one component, $\Lambda$ and $\Lambda'$ are both knots.
\newcommand{\conncmp}[3]{#1 \prec_{#3} #2}

\begin{definition}\label{def:twist-fn}
A \textbf{twist function} on $\mathbf{L}$ is a function $f : \components(\mathbf{L}) \to \halfN$ which satisfies the following conditions for every $k$.
\begin{enumerate}
\item\label{def:twist-fn:0h} If $E_k$ is a 0-handle with $\conncmp{\emptyset}{\Lambda}{E_k}$, then
\[ \begin{aligned} f(\Lambda) &\geq 1. \end{aligned} \]
\item\label{def:twist-fn:wedgeh} If $E_k$ is a $\wedge$-handle with $\conncmp{\Lambda'\cup\Lambda''}{\Lambda}{E_k}$, then
\[ \begin{aligned}
f(\Lambda') &\geq 1/2, & f(\Lambda'') &\geq 1/2, & f(\Lambda') + f(\Lambda'') &= f(\Lambda) + 1.
\end{aligned} \]
\item\label{def:twist-fn:veeh} If $E_k$ is a $\vee$-handle with $\conncmp{\Lambda}{\Lambda'\cup\Lambda''}{E_k}$, then 
\[ \begin{aligned}
f(\Lambda) &\geq 1, & f(\Lambda) &= f(\Lambda') + f(\Lambda'') + 1.
\end{aligned} \]
\item\label{def:twist-fn:final} If $\conncmp{\Lambda}{\Lambda'}{E_k}$ where $\Lambda,\Lambda'$ are knots, then
\[ \begin{aligned} f(\Lambda) &= f(\Lambda'). \end{aligned} \]
\end{enumerate}
A twist function is \textbf{full} when $f(\Lambda)$ is integer for all $\Lambda \in \components(\mathbf{L})$ and $f(\Lambda) \geq 2$ whenever $\Lambda$ is the base of a $\vee$-handle.
A twist function is \textbf{regular} if $f(\Lambda) = 1$ for every $\Lambda \in O(\mathbf{L})$.
\end{definition}

\subsection{From twist functions to cobordisms.}

We can use a twist function on $\mathbf{L}$ to build cobordisms of satellites.

\begin{figure}
\centering
\begin{tikzcd}[column sep=-9.5em, row sep=4em]
  \includegraphics[scale=0.35]{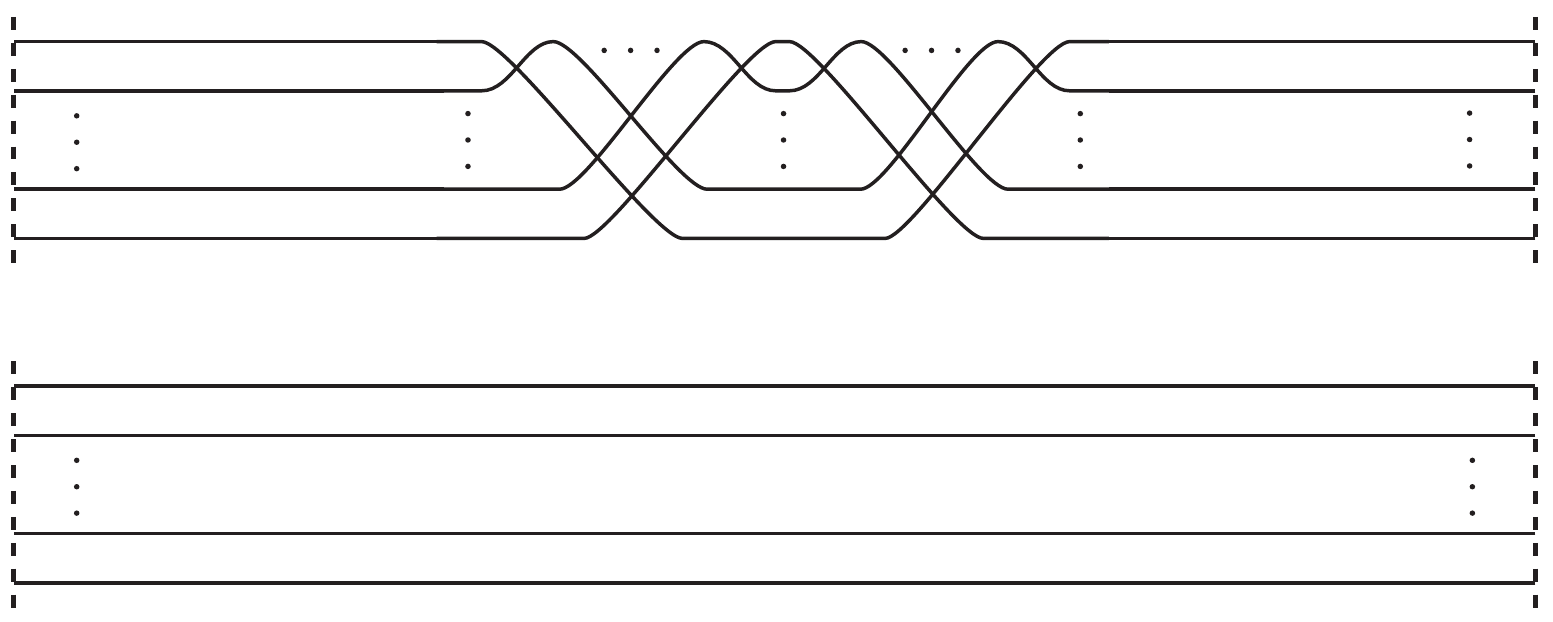}
&&\includegraphics[scale=0.35]{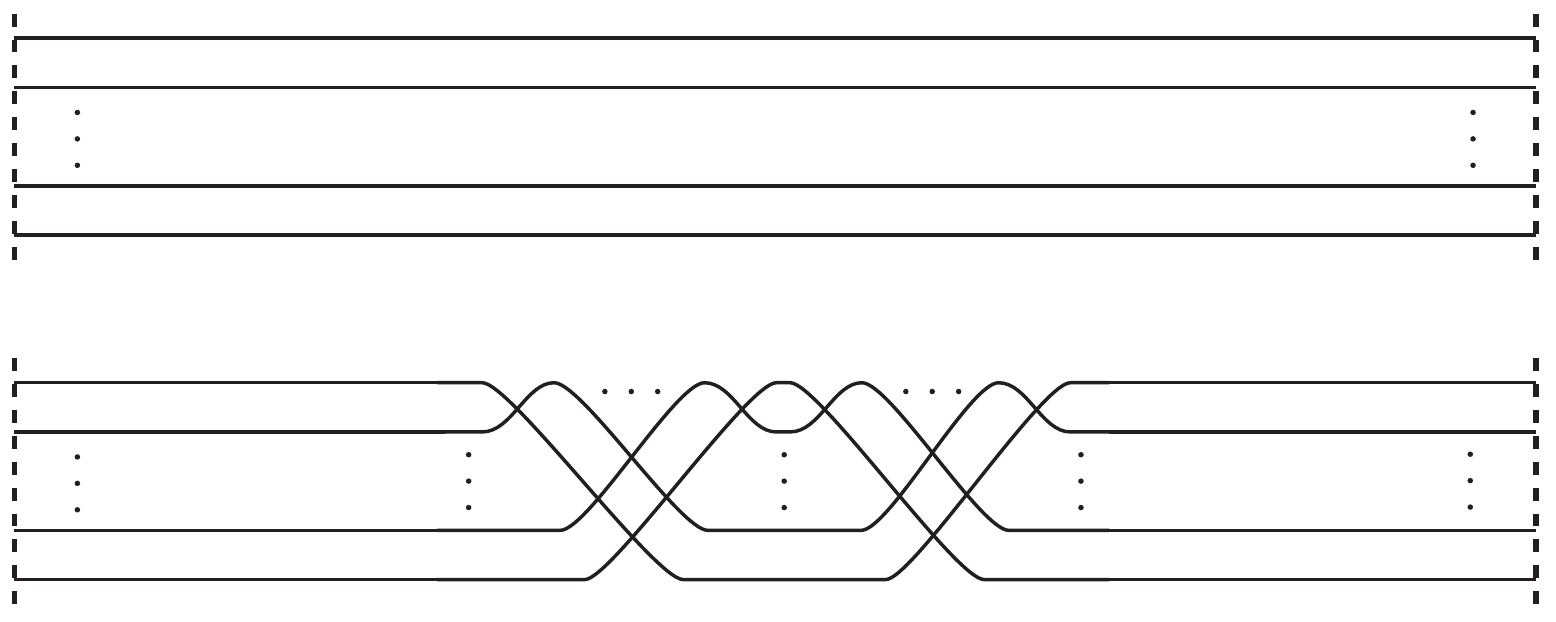}
\\
&&\includegraphics[scale=0.35]{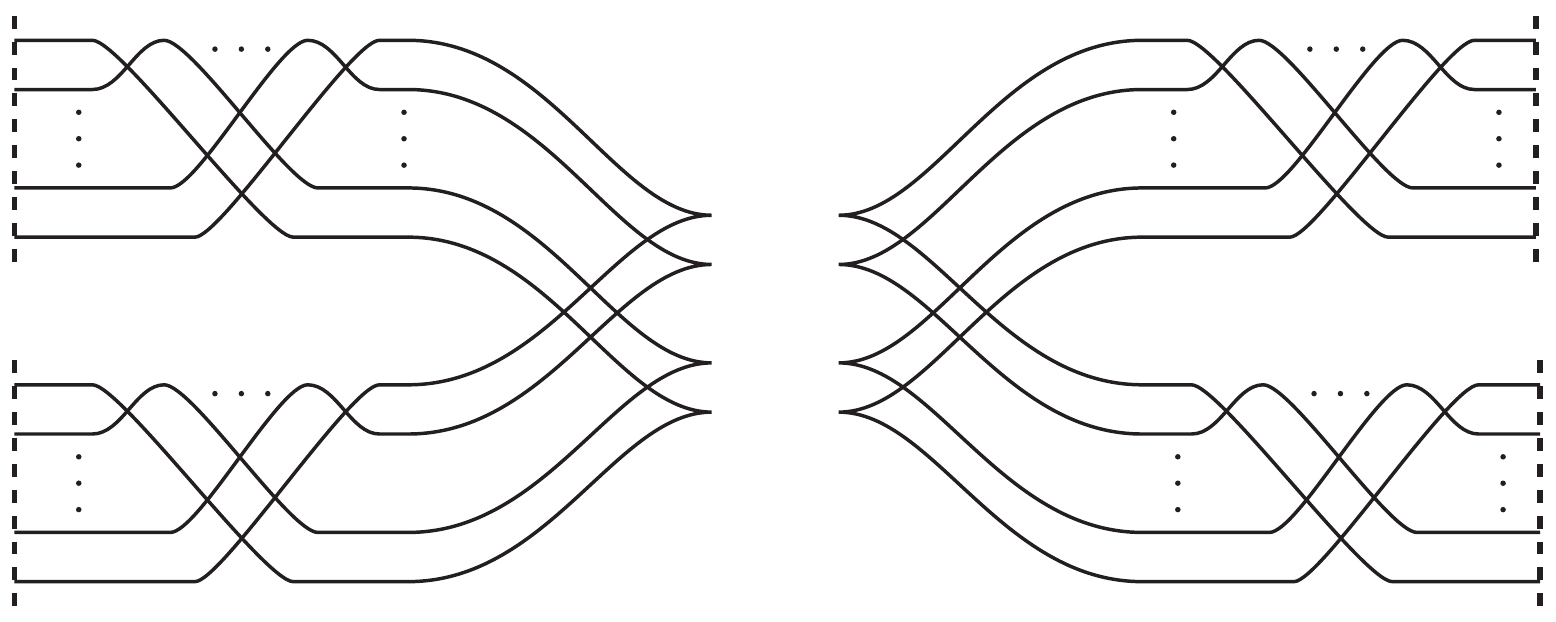}
  \arrow[u,"\text{$\Sigma$-1-handle}"]
\\
& \includegraphics[scale=0.35]{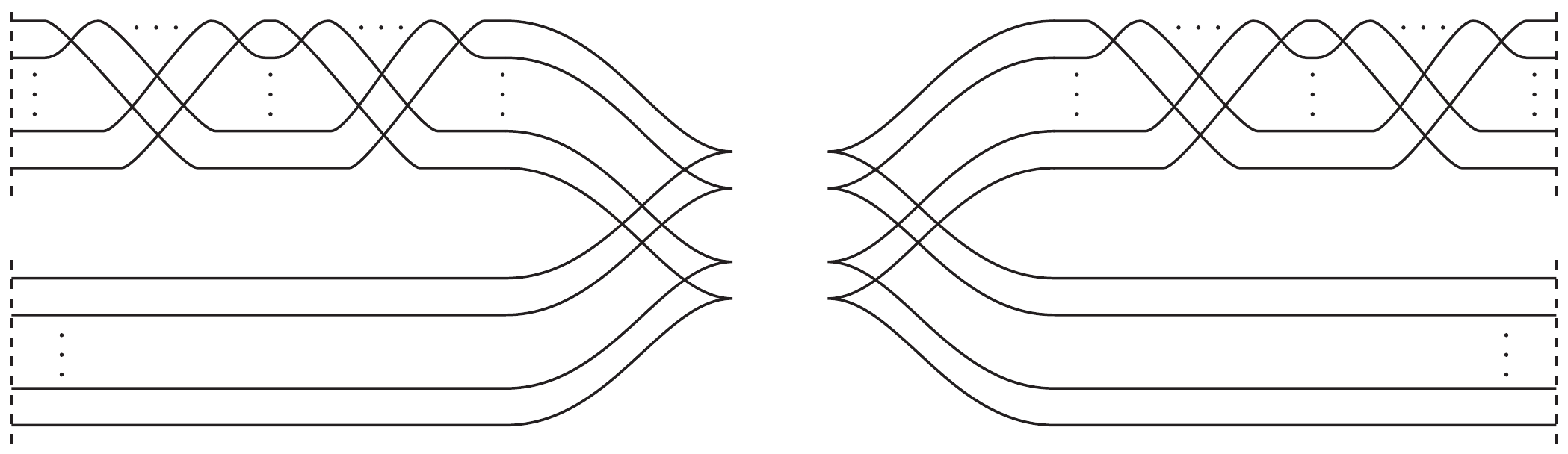}
  \arrow[uul,"\text{$\Sigma$-1-handle}"]
  \arrow[ur,"\text{Isotopy}"]
\end{tikzcd}
\caption{The lower diagram is cobordant to either of the upper diagrams.}
\label{fig:lifted-1h-full}
\end{figure}

\begin{theorem}\label{thm:decompTwist}
Suppose $\Lambda_\pm$ are knots and $\Pi_\pm$ are Legendrian $n$-tangles. Given decomposable cobordisms $\Lambda_-\prec_{\mathbf{L}} \Lambda_+$ and $\Pi_- \prec_{\mathbf{P}} \Pi_+$ as well as a twist function $f$ on $\mathbf{L}$, there exists a decomposable cobordism
\[ \lsat{}{\Lambda_-}{\twist^{f(\Lambda_-)}\Pi_-} \prec \lsat{}{\Lambda_+}{\twist^{f(\Lambda_+)}\Pi_+} \]
if any of the following conditions are met:
\begin{samepage}
\begin{enumerate}
\item $f$ is full and regular;
\item $f$ is regular and $\Pi_\pm$ have symmetric orientation; or 
\item $\Pi_\pm$ have uniform orientation.
\end{enumerate}
\end{samepage}
\end{theorem}

The proof of this theorem is illustrated on a concrete example in Figure~\ref{fig:ex-3-1}.  

We will need to work with \textbf{vertical paths}, i.e. paths in $\mathbf{L} = E_1 \odot \cdots \odot E_m$ that are monotone in the symplectization coordinate $t$. Of course, a vertical path intersects every level $\Lambda_k$ at most once. Notice that any vertical path in $\mathbf{L}$ can be extended downward until it terminates at some component in $\components(\Lambda_-) \cup O(\mathbf{L})$. Since there are no $2$-handles in a decomposable cobordism, the same is true for extending vertical paths upward.  

\begin{figure}
    \centering
    \begin{tikzcd}[column sep=0.1em, row sep=-2.25em]
    \reflectbox{\includegraphics[scale=0.45]{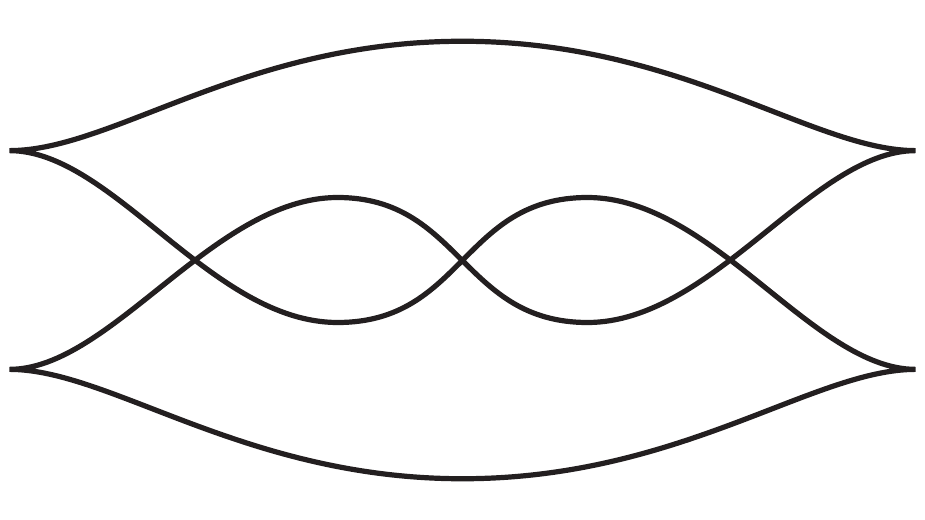}} &
    \begin{gathered} 1\\\vspace{71pt} \end{gathered}\hspace{1.5em} &
    \reflectbox{\includegraphics[scale=0.45]{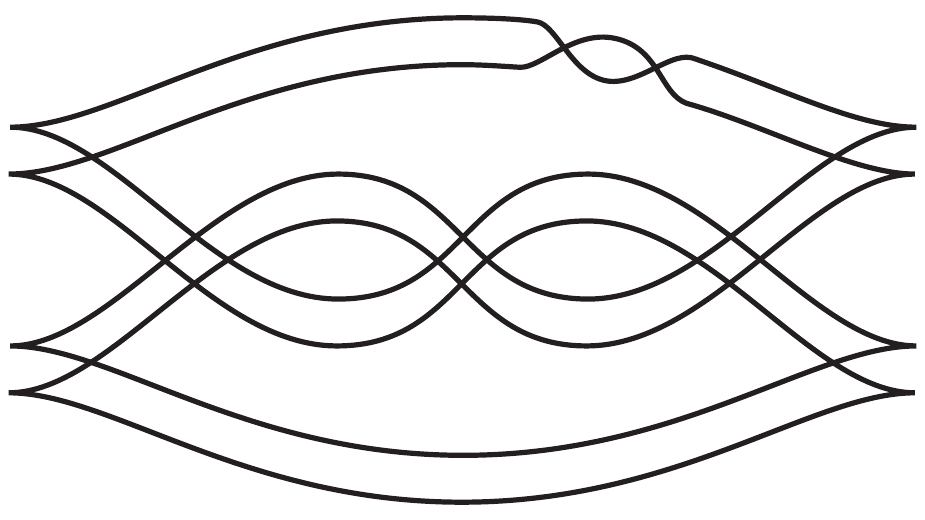}}
    \\
    \reflectbox{\includegraphics[scale=0.45]{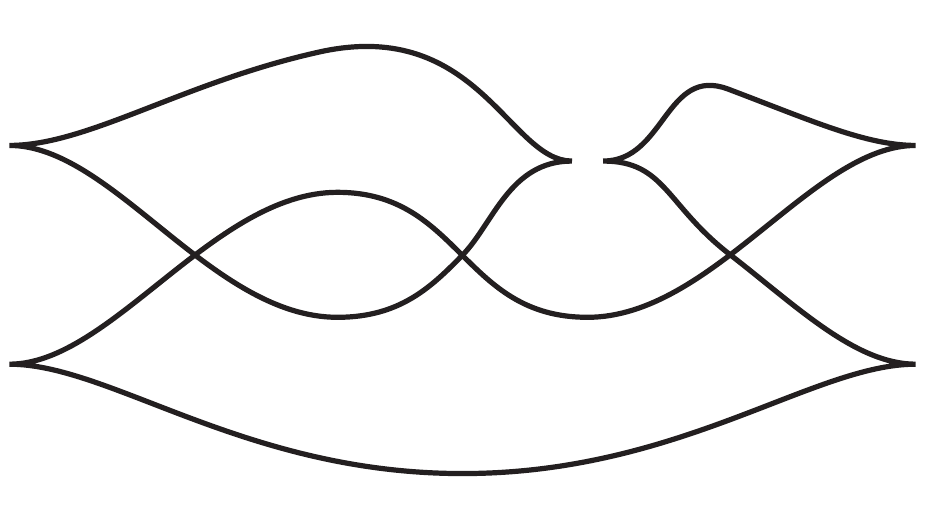}}
        \arrow[u,"\text{$\wedge$-handle}"'] &
    \begin{gathered} 1\vspace{13pt}\\1\\\vspace{45pt} \end{gathered}\hspace{1.5em} &
    \reflectbox{\includegraphics[scale=0.45]{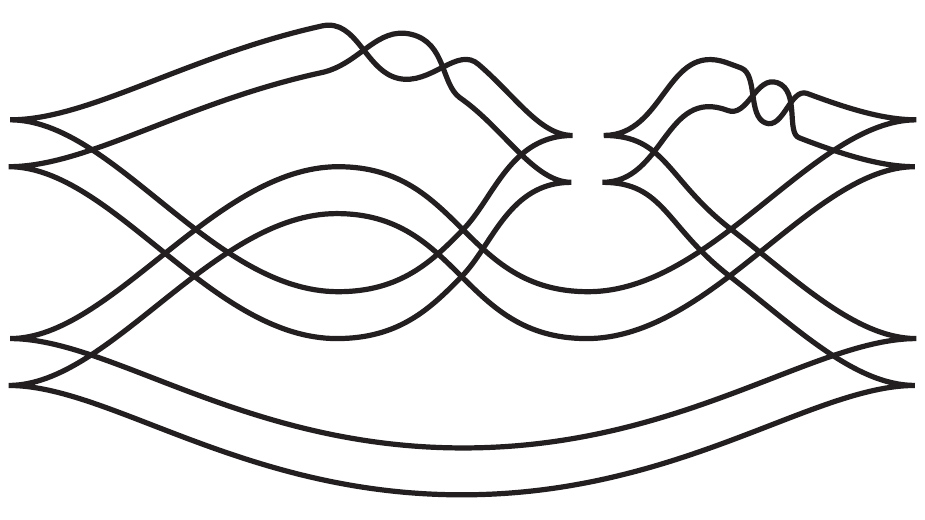}}
        \arrow[u,"\text{Fig.~\ref{fig:lifted-1h-full}}"']
    \\
    \reflectbox{\includegraphics[scale=0.45]{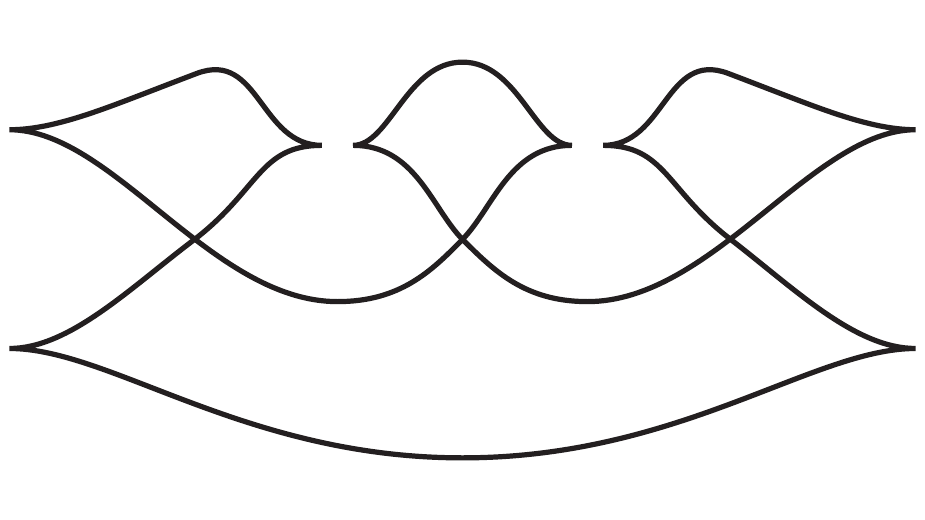}}
        \arrow[u,"\text{$\vee$-handle}"'] &
    \begin{gathered} 3\\\vspace{77pt} \end{gathered}\hspace{1.5em} &
    \reflectbox{\includegraphics[scale=0.45]{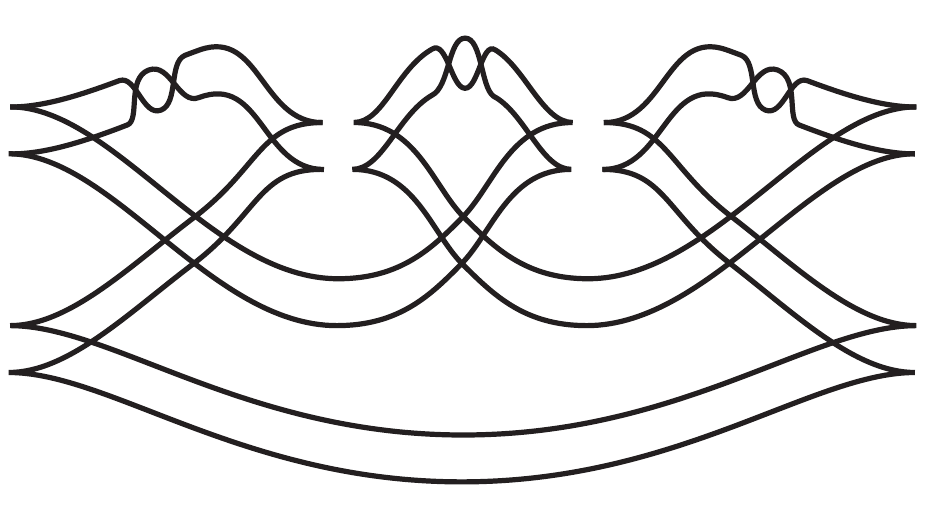}}
        \arrow[u,"\text{Fig.~\ref{fig:lifted-1h-full}}"']
    \\
    \reflectbox{\includegraphics[scale=0.45]{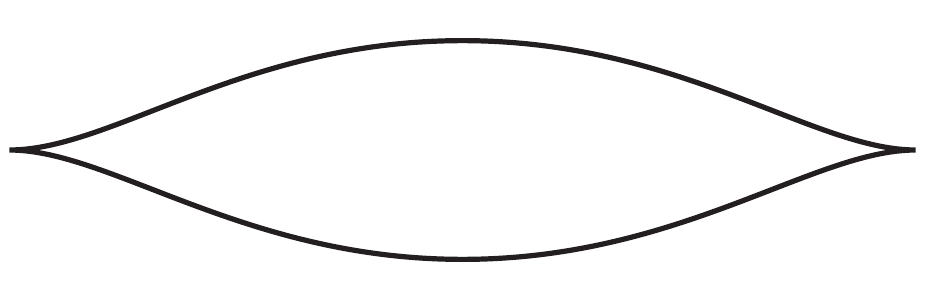}}
        \arrow[u,"\text{Isotopy}"'] &
    \begin{gathered} 3\\\vspace{14pt} \end{gathered}\hspace{1.5em} &
    \reflectbox{\includegraphics[scale=0.45]{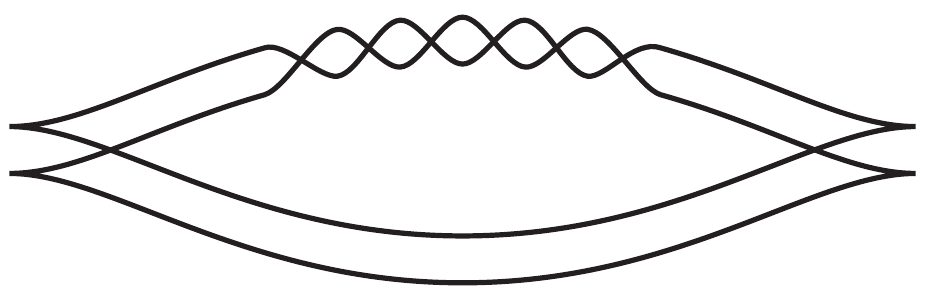}}
        \arrow[u,"\text{Isotopy}"']
    \\
    \end{tikzcd}
    \caption{On the left, a decomposable cobordism $\Upsilon \prec 3_1$ with the values of a full and regular twist function shown next to each component. On the right, Theorem~\ref{thm:decompTwist}~(1) applied to this twist function, resulting in a decomposable cobordism $\lsat{}{\Upsilon}{\twist^3} \prec \lsat{}{3_1}{\twist}$.}
    \label{fig:ex-3-1}
\end{figure}

\begin{proof}
Our first goal is to build a cobordism \[ \lsat{}{\Lambda_-}{\twist^{f(\Lambda_-)}_s} \prec_{\Sigma(\mathbf{L})} \lsat{}{\Lambda_+}{\twist^{f(\Lambda_+)}_s} \]
where $s$ agrees with the orientation of $\Pi_\pm$ at the boundary. To do this, we first take the $n$-copy of each $\Lambda \in \components(\mathbf{L})$. We then apply $f(\Lambda)$ twists to each  such  $n$-copy. Notice that an isotopy in ${\mathbf{L}}$ immediately yields an isotopy in the $n$-copy. 
For any $0$-handle in ${\mathbf{L}}$, Lemma \ref{lem:lifted-0h} ensures that we can get a $\Sigma$-$0$-handle if either $s$ is uniform or $f$ is regular.
For any $1$-handle in ${\mathbf{L}}$, we can obtain a $\Sigma$-$1$-handle if either $s$ is symmetric (in which case we use the diagram of Figure \ref{fig:lifted-1h-half}) or if $f$ is full (in which case we use the diagram of Figure \ref{fig:lifted-1h-full}).
This process yields a decomposable cobordism
\[ \lsat{}{\Lambda_-}{\twist^{f(\Lambda_-)}} \prec_{\Sigma(\mathbf{L})} \lsat{}{\Lambda_+}{\twist^{f(\Lambda_+)}}. \]

To obtain a cobordism
\[ \lsat{}{\Lambda_-}{\twist^{f(\Lambda_-)}\Pi_-} \prec_{\Sigma(\mathbf{L})} \lsat{}{\Lambda_+}{\twist^{f(\Lambda_+)}\Pi_-}, \]
we first choose a vertical path $\gamma\subset \mathbf{L}$ from $\Lambda_-$ to $\Lambda_+$ that does not pass through any $1$-handle where it is attached.  Next, place the tangle $\Pi_-$ into the $n$-copy of $\Lambda_k$ where $\Lambda_k$ intersects $\gamma$. Lastly, we attach the elementary cobordisms that make up $P$ in a small neighborhood around $\Pi_-$ in the $n$-copy of $\Lambda_+$ to obtain the desired cobordism.
\end{proof}

The following immediate corollary is useful for removing twists on $\Lambda_-$ in the case of uniform orientation. It may be thought of as a decomposable analogue of the technique used in the proof of Proposition~\ref{main_geom_nonorientable}.

\begin{corollary}\label{cor:remove-twists}
Given a Legendrian knot $\Lambda$ and a Legendrian $n$-tangle $\Pi$, if $\Pi$ has uniform orientation then for any $k \in \halfN$
there exists a cobordism
 \[ \lsat{}{\Lambda}{\twist^{1/2}\Pi} \prec
    \lsat{}{\Lambda}{\twist^{1/2+k}\Pi}. \]
\end{corollary}
\begin{proof}
Let $\mathbf{L} = E_1 \odot E_2$ where $\Lambda \prec_{E_1} \Lambda \cup \Upsilon$ is a 0-handle and $\Lambda \cup \Upsilon \prec_{E_2} \Lambda'$ is a $\wedge$-handle trivially joining $\Lambda$ and $\Upsilon$, resulting in a knot $\Lambda'$ isotopic to $\Lambda$. Apply Theorem~\ref{thm:decompTwist}~(3) to the twist function $f$ on $\mathbf{L}$ given by $f(\Lambda) = 1/2$, $f(\Upsilon) = 3/2+k$, and $f(\Lambda') = 1/2+k$.
\end{proof}

\subsection{Existence of twist functions.}
\label{ssec:exists-twist}

While twist functions always exist, regular ones may not.  In order to build regular twist functions for a decomposable cobordism $\mathbf{L}$, we use Property A, which we recall here using language developed in this section.

\begin{definition}[Property A, Rewritten]
A decomposable cobordism $\mathbf{L}$ satisfies \allbf{Property A} if, for every $\vee$-handle $E_k$ in $\mathbf{L}$,
there exists a downward vertical path in $\mathbf{L}$ which starts at $\base(E_k)$ and ends at some component of $\Lambda_-$.
\end{definition}

If $\mathbf{L}$ contains no $\vee$-handles or no $0$-handles, then $\mathbf{L}$ trivially satisfies Property A. In particular, it is easy to check that $\mathbf{L}$ is a concordance if and only if there are no $\vee$-handles, so in this case Property A holds. As mentioned in the introduction, in many cases, one can deform $\mathbf{L}$ to ensure that it satisfies Property A (cf.\ Figure~\ref{fig:no-property-D}), but the authors are unsure if this can  be done for any cobordism between knots.

\begin{figure}
    \centering
    \begin{gather*}
    \begin{tikzcd}[column sep=0.1em, row sep=2.25em]
    \includegraphics[scale=0.45]{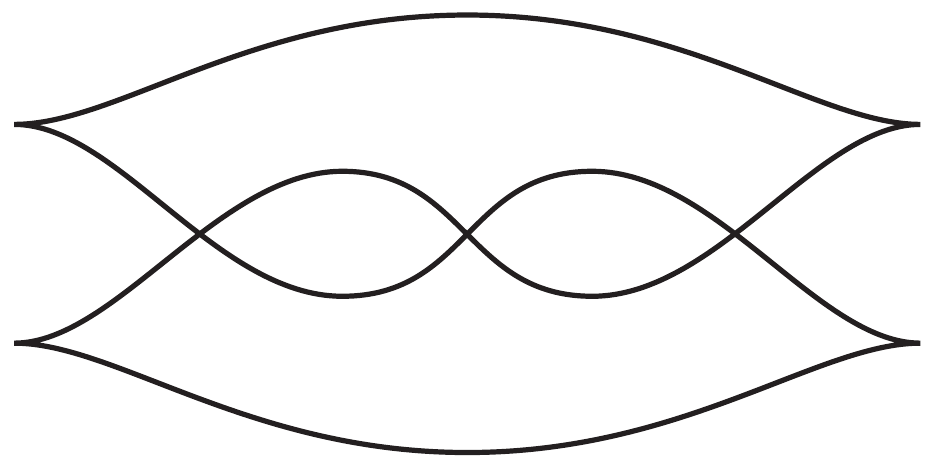} \\
    \includegraphics[scale=0.45]{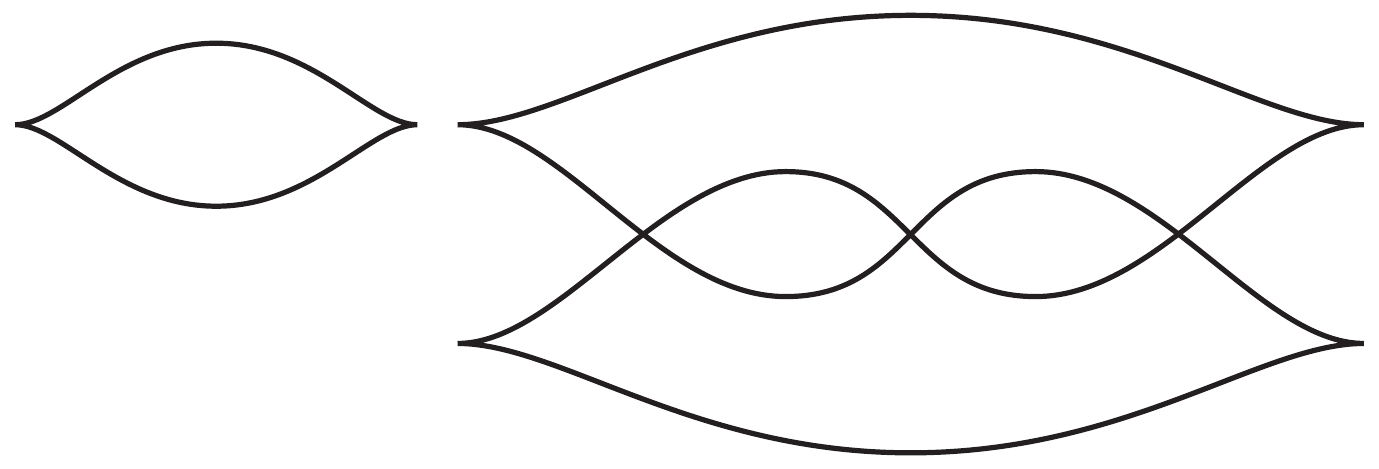}
        \arrow[u,"\substack{\text{Isotopy\,}\\\text{$\wedge$-handle}}"'] \\
    \includegraphics[scale=0.45]{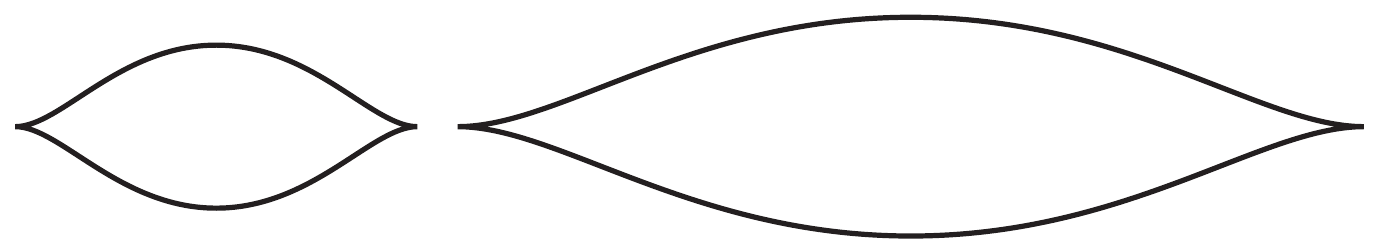}
        \arrow[u,"\text{Fig.~\ref{fig:ex-3-1}}"'] \\
    \includegraphics[scale=0.45]{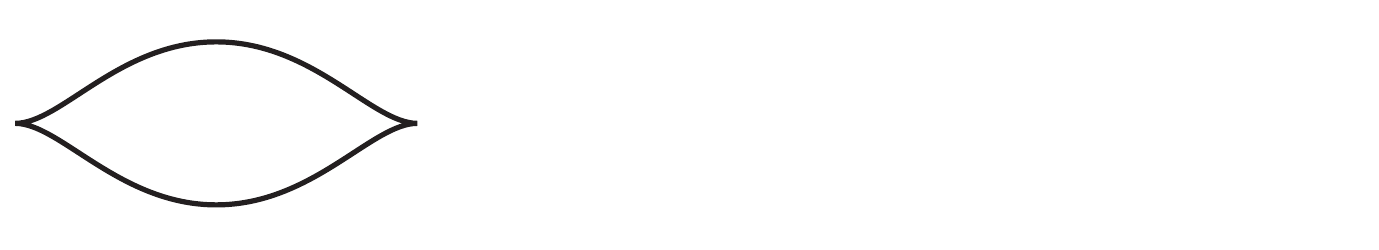}
        \arrow[u,"\text{0-handle}"']
    \end{tikzcd} 
    \end{gather*}
    \caption{A decomposable cobordism $\Upsilon \prec 3_1$ which does not satisfy Property A. This cobordism is a deformation of the decomposable cobordism $\Upsilon \prec 3_1$ shown on the left of Figure~\ref{fig:ex-3-1}, which does satisfy Property A.}
    \label{fig:no-property-D}
\end{figure}

We are now ready to build some twist functions. See Figures~\ref{fig:ex-3-1} and \ref{fig:ex-9-46} for examples of this construction in action.

\begin{proposition}\label{prop:exists-tw-fn}
For a decomposable cobordism of knots $\Lambda_-\prec_\mathbf{L}\Lambda_+$,
\begin{enumerate}
    \item there exists a full twist function $f$ on $\mathbf{L}$ with $f(\Lambda_+) = 1$, and
    \item there exists a (non-full) twist function $f$ on $\mathbf{L}$ with $f(\Lambda_+) = 1/2$.
\end{enumerate}
Moreover, the twist function in either case is regular if and only if $\mathbf{L}$ satisfies Property A, in which case $f(\Lambda_-) = f(\Lambda_+) + 2g(\mathbf{L}).$
\end{proposition}

\begin{proof}
To prove part (1), we proceed by induction on the number of elementary cobordisms in $\mathbf{L}=E_1 \odot \cdots \odot E_n$. In particular, we prove that there exists a full twist function $f_k$ on $\mathbf{L}_k=E_1 \odot \cdots \odot E_k$ such that $f(\Lambda) = 1$ for every $\Lambda \in \components(\Lambda_k)$. The base case of $k = 0$ is trivial, so suppose there exists a full twist function $f_k$ on $\mathbf{L}_k$ such that $f(\Lambda) = 1$ for every $\Lambda \in \components(\Lambda_k)$. We want to show that there exists a full twist function $f_{k+1}$ on $\mathbf{L}_{k+1}$ such that $f(\Lambda) = 1$ for every $\Lambda \in \components(\Lambda_{k+1})$.
If $E_{k+1}$ is an isotopy, a 0-handle, or a $\wedge$-handle, set $f_{k+1}(\Lambda) = 1$ for every $\Lambda \in \components(\Lambda_{k+1})$ and $f_{k+1}(\Lambda) = f_k(\Lambda)$ otherwise. 
If $E_{k+1}$ is a $\vee$-handle, choose a downward vertical path $\gamma$ starting at the base of $E_{k+1}$. Define $f_{k+1}(\Lambda) = 1$ for every $\Lambda \in \components(\Lambda_{k+1})$, $f_{k+1}(\Lambda) = f_k(\Lambda) + 2$ for every $\Lambda \in \components(\mathbf{L}_k)$ which intersects $\gamma$, and $f_{k+1}(\Lambda) = f_k(\Lambda)$ otherwise. One can check that such  $f_{k+1}$ is still a full twist function, completing the inductive step.

Notice that if $\mathbf{L}$ satisfies Property A, we can make sure that the chosen path $\gamma$ always terminates at $\Lambda_-$. This ensures that the resulting $f$ is regular.
Conversely, assume that we have built a full, regular $f$. For any $\vee$-handle we have $f(\base(E_k)) \geq 2$. There must exist a downward path $\gamma$, starting at $\base(E_k)$ such that, for any $\Lambda \in \components(\mathbf{L})$ intersecting $\gamma$, $f(\Lambda) \geq f(\base(E_k)) \geq 2$. Since $f$ is regular, $\gamma$ must terminate at $\Lambda_-$. So $\mathbf{L}$ satisfies Property A.

To prove part (2), let $f_1$ be the full twist function built in part (1). Consider an upward vertical path $\gamma\subset\mathbf{L}$ from $\Lambda_-$ to $\Lambda_+$.  Define $f(\Lambda) = f_1(\Lambda) - 1/2$ for every $\Lambda \in\components(\mathbf{L})$ which intersects $\gamma$ and $f(\Lambda) = f_1(\Lambda)$ otherwise. It is straightforward to verify that this $f$ is indeed a twist function. Furthermore, this $f$ is regular if and only if $f_1$ is regular.

The final claim follows from the fact that $2g(\mathbf{L})$ may be calculated as the difference between the number of $1$-handles and the number of $0$-handles.
\end{proof}

\begin{remark}\label{rmk:full-f>1}
Let $\Lambda_-\prec_\mathbf{L}\Lambda_+$ be a cobordism between knots which contains at least one $0$- or $1$-handle. Then $\mathbf{L}$ contains at least one $\wedge$-handle. It follows that for any full twist function $f$ on $\mathbf{L}$, we have $f(\Lambda_+) \geq 1$.
\end{remark}

For general (non-full) twist functions on cobordisms between knots, we can sometimes, but not always, ensure that $f(\Lambda_+)=0$ in parallel to the construction of satellites of concordances in \cite{cns:obstr-concordance}.

\begin{example}\label{ex:zero-at-top}
If $\mathbf{L}$ is a cobordism between knots which contains no 0-handles, then there exists a (regular, but non-full) twist function $f$ on $\mathbf{L}$ with $f(\Lambda_+) = 0$. The construction is very similar to that in Proposition~\ref{prop:exists-tw-fn}~(1) with the following main modifications. In the case of a $\wedge$-handle $E_k$, pick two downwards vertical paths terminating at $\Lambda_-$, $\gamma_1$ starting at $\base_1(E_k)$ and $\gamma_2$ starting at $\base_2(E_k)$, and define $f_{k+1}(\Lambda) = f_k(\Lambda) + 1$ for every $\Lambda \in \components(\mathbf{L}_k)$ which intersects both $\gamma_1$ and $\gamma_2$, $f_{k+1}(\Lambda) = f_k(\Lambda) + 1/2$ for every $\Lambda \in \components(\mathbf{L}_k)$ which intersects just $\gamma_1$ or just $\gamma_2$, and $f_{k+1}(\Lambda) = f_k(\Lambda)$ otherwise. In the case of a $\vee$-handle $E_k$, define $f_{k+1}(\Lambda) = f_k(\Lambda) + 1$ for every $\Lambda \in \components(\mathbf{L}_k)$ which intersects $\gamma$, instead of $f_{k+1}(\Lambda) = f_k(\Lambda) + 2$.
\end{example}

\begin{example}
The cobordism $\Upsilon \prec m(9_{46})$ pictured on the left of Figure~\ref{fig:ex-9-46} does not admit a twist function $f$ with $f(\Lambda_+) = 0$. The twist function shown in Figure~\ref{fig:ex-9-46} has the smallest possible value of $f(\Lambda_+)$. 
\end{example}

\begin{figure}
    \centering
    \begin{tikzcd}[column sep=0.1em, row sep=-2.25em]
    \reflectbox{\includegraphics[scale=0.45]{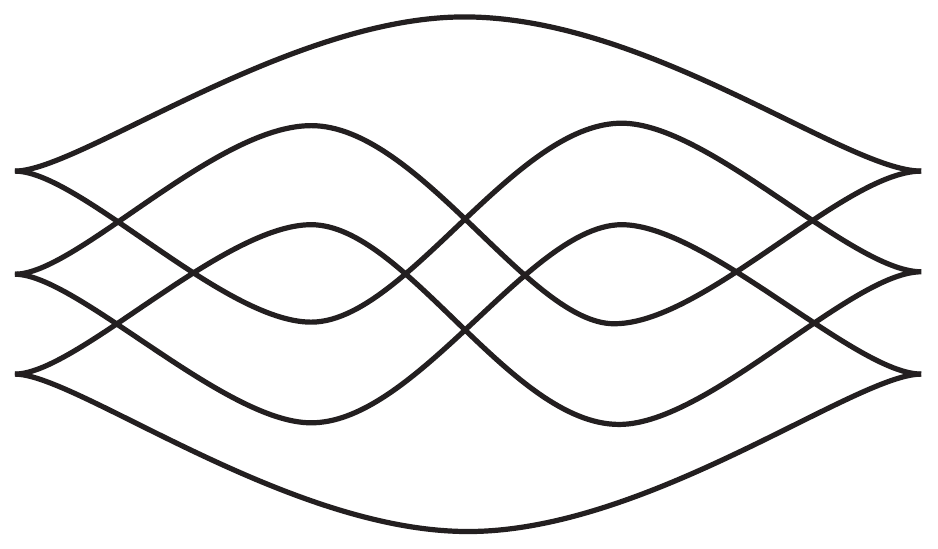}} &
    \begin{gathered} 1/2\\\vspace{73pt} \end{gathered}\hspace{1.5em} &
    \reflectbox{\includegraphics[scale=0.45]{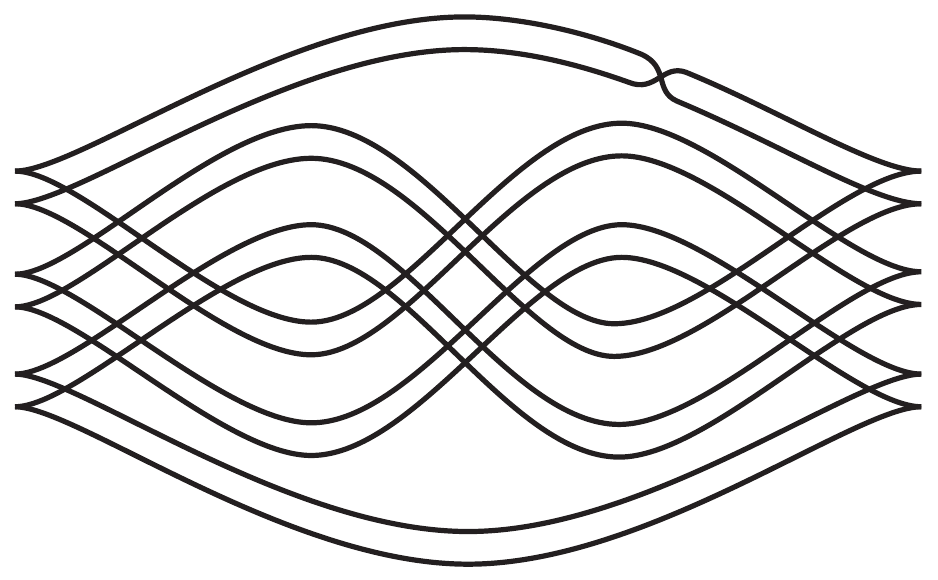}}
    \\
    \reflectbox{\includegraphics[scale=0.45]{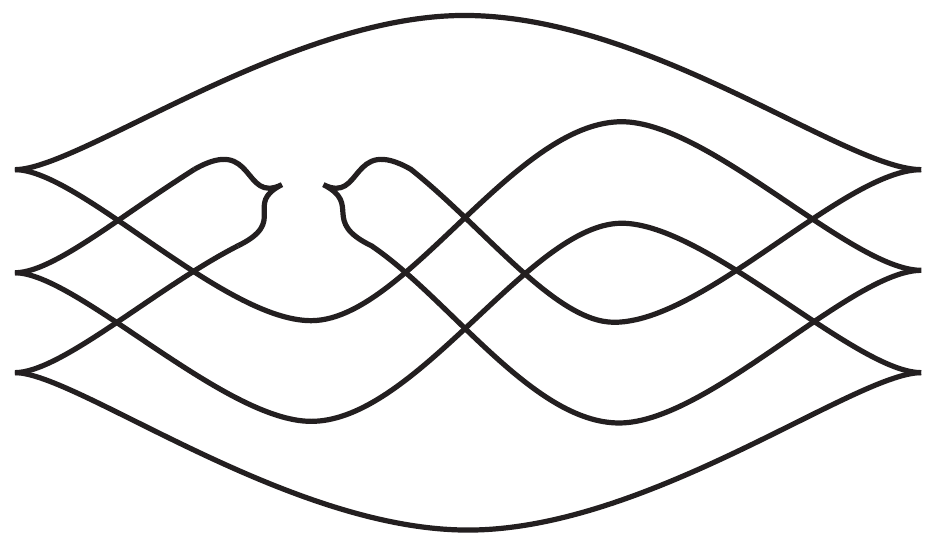}}
        \arrow[u,"\text{$\wedge$-handle}"'] &
    \begin{gathered} 1\vspace{10pt}\\1/2\\\vspace{48pt} \end{gathered}\hspace{1.5em} &
    \reflectbox{\includegraphics[scale=0.45]{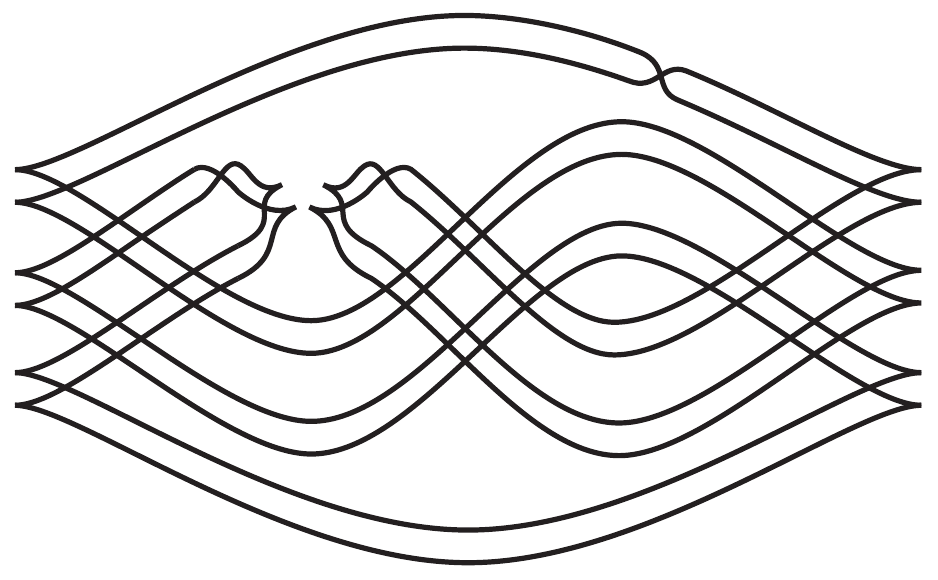}}
        \arrow[u,"\text{$\Sigma$-$\wedge$-handle}"']
    \\
    \reflectbox{\includegraphics[scale=0.45]{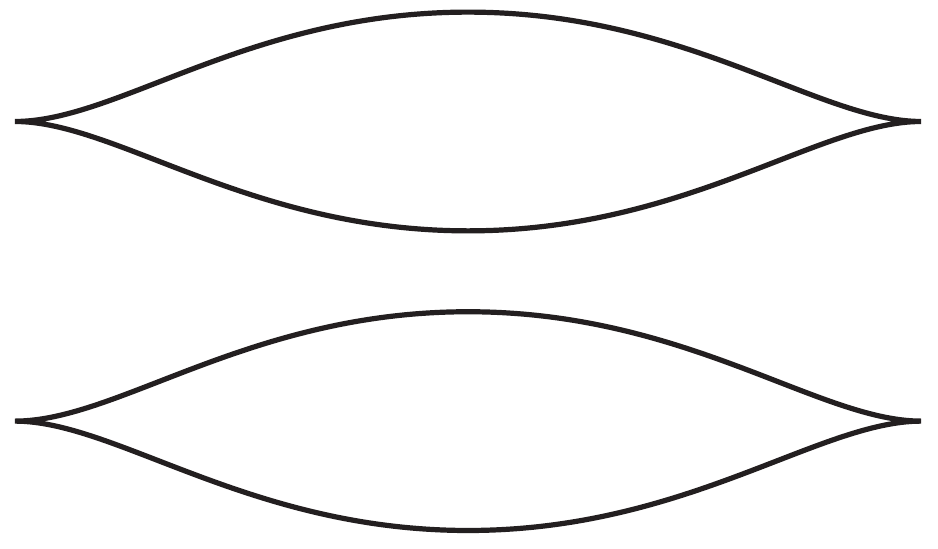}}
        \arrow[u,"\text{Isotopy}"'] &
    \begin{gathered} 1\vspace{23pt}\\1/2\\\vspace{44pt} \end{gathered}\hspace{1.5em} &
    \reflectbox{\includegraphics[scale=0.45]{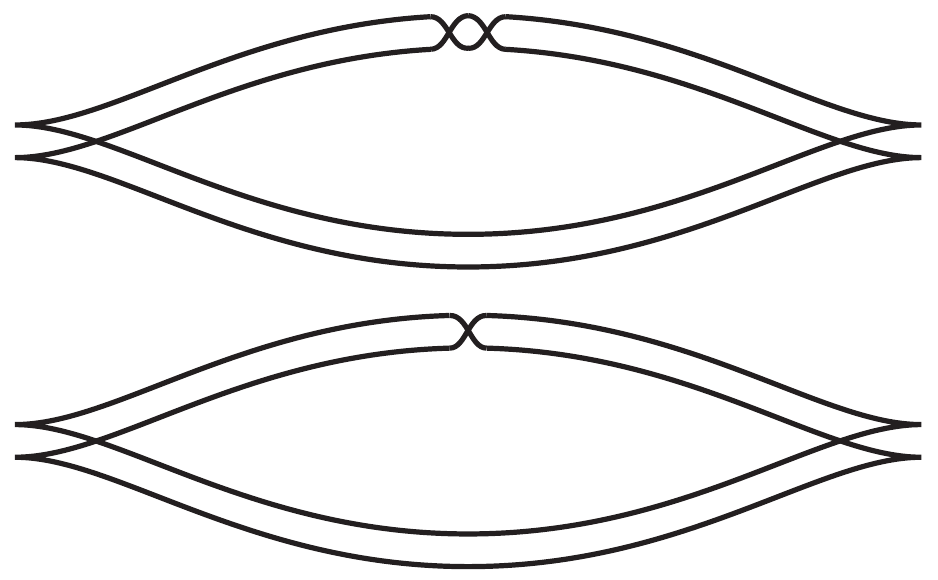}}
        \arrow[u,"\text{Isotopy}"']
    \\
    \reflectbox{\includegraphics[scale=0.45]{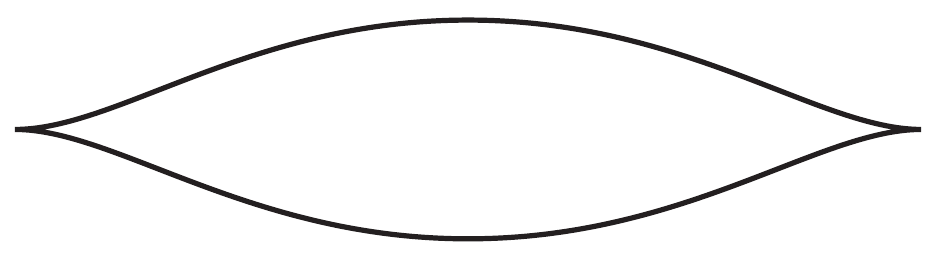}}
        \arrow[u,"\text{0-handle}"'] &
    \begin{gathered} 1/2\\\vspace{9pt} \end{gathered}\hspace{1.5em} &
    \reflectbox{\includegraphics[scale=0.45]{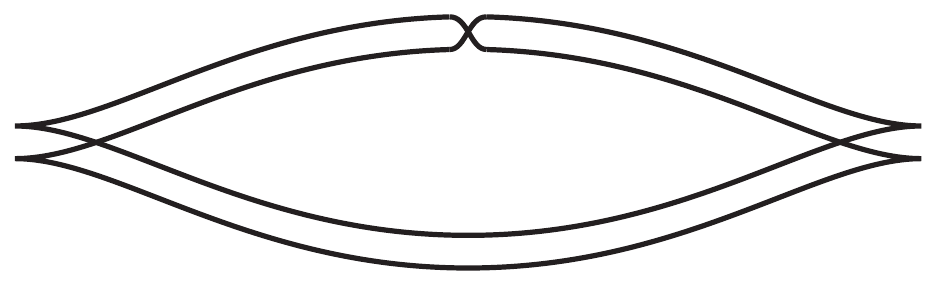}}
        \arrow[u,"\text{$\Sigma$-0-handle}"']
    \\
    \end{tikzcd}
    \caption{On the left, a decomposable cobordism $\Upsilon \prec m(9_{46})$ and the value of its regular twist function from Proposition~\ref{prop:exists-tw-fn}~(2) shown next to each component. On the right, Theorem~\ref{thm:decompTwist}~(2) applied to this twist function, resulting in a cobordism $\lsat{}{\Upsilon}{\twist_s^{1/2}} \prec \lsat{}{m(9_{46})}{\twist_s^{1/2}}$ for $s$ symmetric.}    \label{fig:ex-9-46}
\end{figure}

\begin{example}
The cobordism $\Upsilon \prec 3_1\,\#\,m(9_{46})$ pictured on the left of Figure~\ref{fig:ex-3-1-sum-9-46} contains a 0-handle but admits a (regular) twist function $f$ with $f(\Lambda_+) = 0$. The upper half of this twist function comes from Example~\ref{ex:zero-at-top} and the lower half from the twist function in Figure~\ref{fig:ex-9-46}. As a consequence, Example~\ref{ex:zero-at-top} does not cover every case where there exists a twist function $f$ with $f(\Lambda_+) = 0$.
\end{example}

Notice that we can never have $f(\Lambda_-)=0$ unless $\mathbf{L}$ comes from an isotopy.  To see this, let $\mathbf{L}$ be a cobordism between knots which contains at least one 0- or 1-handle. By similar reasoning as in Remark~\ref{rmk:full-f>1}, it follows that for any twist function $f$ on $\mathbf{L}$, $f(\Lambda_-) > 0$.

\begin{figure}
    \centering
    \begin{tikzcd}[column sep=0.1em, row sep=-2.25em]
    \reflectbox{\includegraphics[scale=0.45]{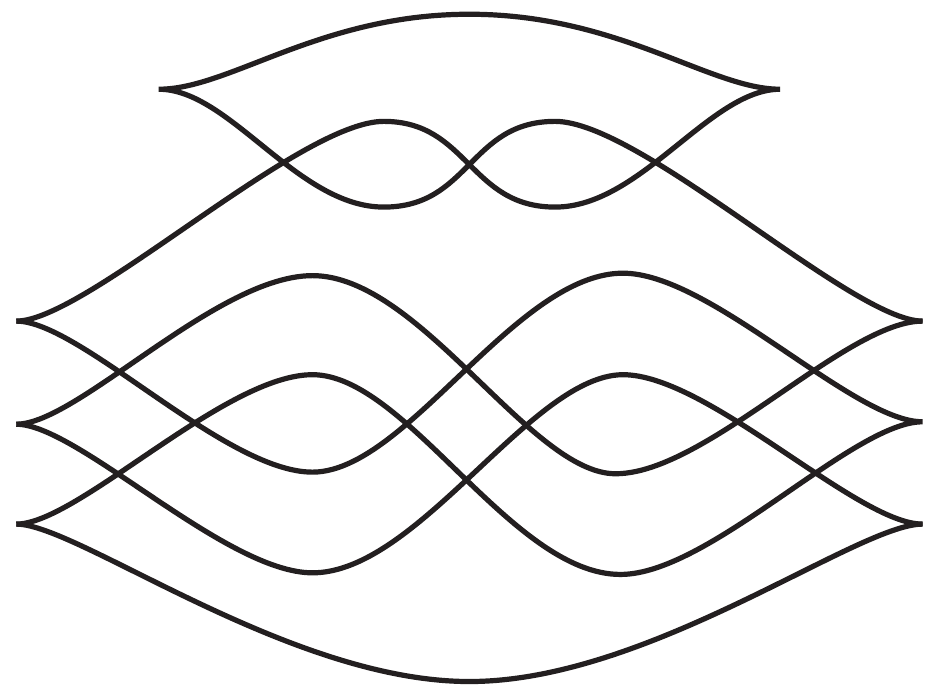}} &
    \begin{gathered} 0\\\vspace{73pt} \end{gathered}\hspace{1.5em} &
    \reflectbox{\includegraphics[scale=0.45]{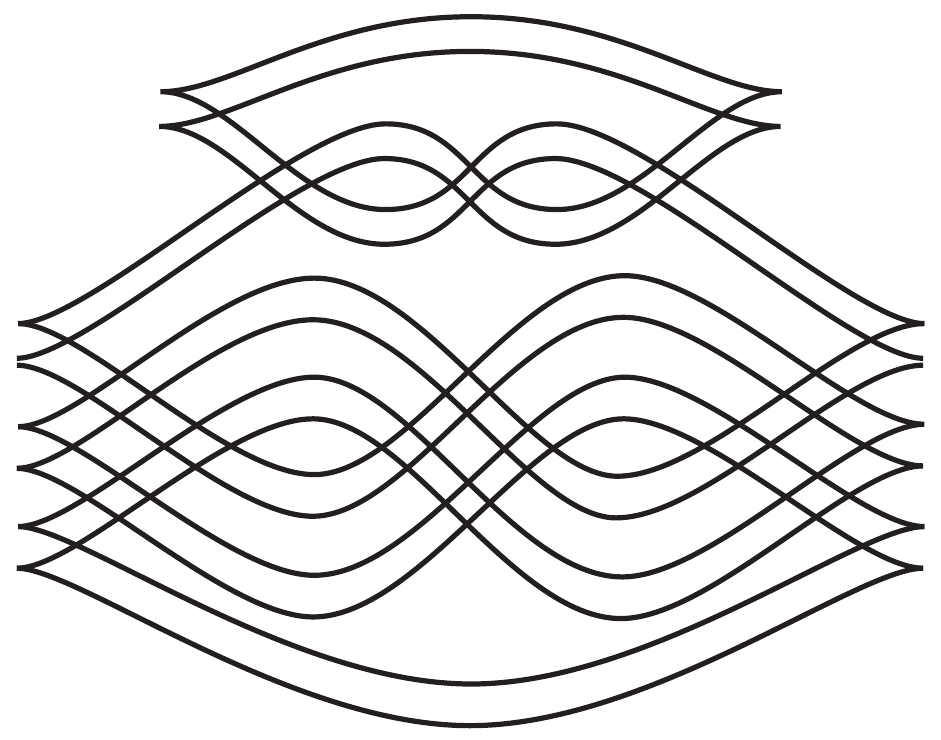}}
    \\
    \reflectbox{\includegraphics[scale=0.45]{graphics/ex-9-46-1.pdf}}
        \arrow[u,"\text{Fig.~\ref{fig:ex-3-1}}"'] &
    \begin{gathered} 2\\\vspace{73pt}\end{gathered}\hspace{1.5em} &
    \reflectbox{\includegraphics[scale=0.45]{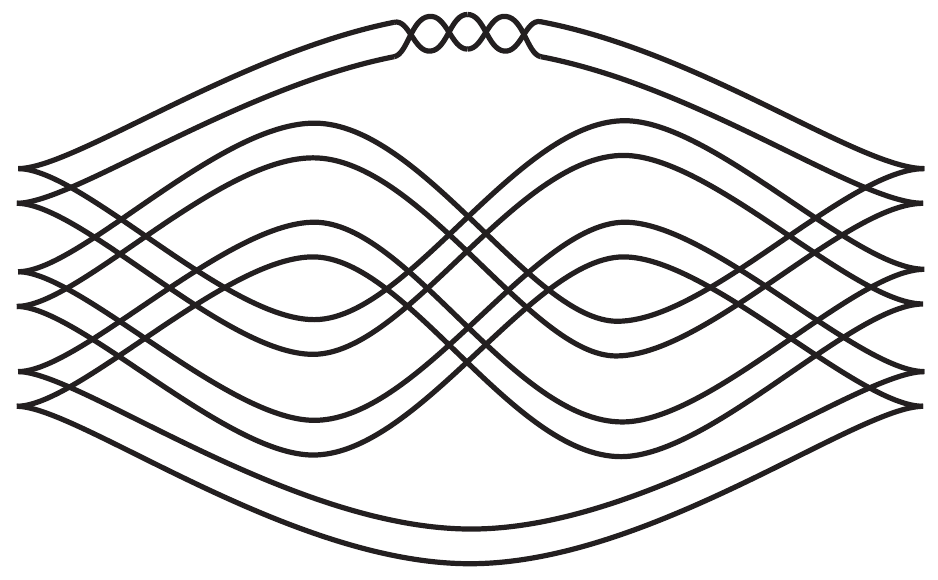}}
        \arrow[u]
    \\
    \reflectbox{\includegraphics[scale=0.45]{graphics/ex-9-46-4.pdf}}
        \arrow[u,"\text{Fig.~\ref{fig:ex-9-46}}"'] &
    \begin{gathered} 2\\\vspace{9pt} \end{gathered}\hspace{1.5em} &
    \reflectbox{\includegraphics[scale=0.45]{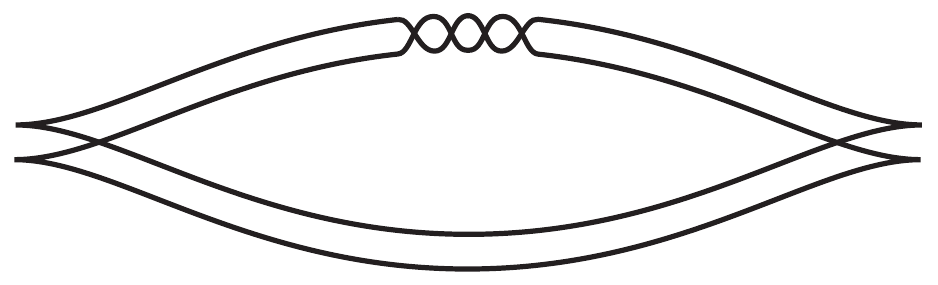}}
        \arrow[u]
    \\
    \end{tikzcd}
    \caption{On the left, a decomposable cobordism $\Upsilon \prec 3_1\,\#\,m(9_{46})$ and a twist function on it with $f(\Lambda_+) = 0$. 
    On the right, Theorem~\ref{thm:decompTwist}~(2) applied to this twist function, resulting in a cobordism $\lsat{}{\Upsilon}{\twist_s^2} \prec \lsat{}{3_1\,\#\,m(9_{46})}{\twist^0_s}$ for $s$ symmetric.}
    \label{fig:ex-3-1-sum-9-46}
\end{figure}

\subsection{Decomposable cobordisms of satellites}

The following result is a more general version of Theorem~\ref{thm:decomp}, where the possibly non-orientable cobordism in the statement of Theorem~\ref{thm:decomp} comes from ignoring orientation in \eqref{cor:decompGen:uniform}.

\begin{corollary}\label{cor:decompGen}
Suppose $\Lambda_\pm$ are knots and $\Pi_\pm$ are Legendrian $n$-tangles. Given decomposable cobordisms $\Lambda_- \prec_\mathbf{L} \Lambda_+$ and $\Pi_- \prec_{\mathbf{P}} \Pi_+$,
\begin{enumerate}
\item\label{cor:decompGen:full} If $\mathbf{L}$ satisfies Property A, then there exists a decomposable cobordism
 \[ \lsat{}{\Lambda_-}{\twist^{1+2g(\mathbf{L})}\Pi_-} \prec
    \lsat{}{\Lambda_+}{\twist\Pi_+}; \]
\item\label{cor:decompGen:half} If $\mathbf{L}$ satisfies Property A and $\Pi_\pm$ have symmetric orientation, then there exists a decomposable cobordism
 \[ \lsat{}{\Lambda_-}{\twist^{1/2+2g(\mathbf{L})}\Pi_-} \prec
    \lsat{}{\Lambda_+}{\twist^{1/2}\Pi_+}; \]
\item\label{cor:decompGen:uniform} If $\Pi_\pm$ have uniform orientation, then there exists a decomposable cobordism
 \[ \lsat{}{\Lambda_-}{\twist^{1/2}\Pi_-} \prec
    \lsat{}{\Lambda_+}{\twist^{1/2}\Pi_+}. \]
\end{enumerate}
\end{corollary}

\begin{proof}
\leavevmode\makeatletter\@nobreaktrue\makeatother
\begin{enumerate}
\item Apply Theorem~\ref{thm:decompTwist}~(1) to the full and regular twist function from Proposition~\ref{prop:exists-tw-fn}~(1). 

\item Apply Theorem~\ref{thm:decompTwist}~(2) to the regular twist function from Proposition~\ref{prop:exists-tw-fn}~(2).

\item Apply Theorem~\ref{thm:decompTwist}~(3) to the twist function $f$ from Proposition~\ref{prop:exists-tw-fn}~(2) to get a cobordism
\[ \lsat{}{\Lambda_-}{\twist^{f(\Lambda_-)}\Pi_-} \prec \lsat{}{\Lambda_+}{\twist^{1/2}\Pi_+}, \]
then apply Corollary~\ref{cor:remove-twists} with $k = f(\Lambda_-) - 1/2$.
\end{enumerate}
\end{proof}

\section{Connected sums and higher dimensions}
\label{sec:connect-sum}

In this section, we show how to apply Theorem \ref{thm:geom} to construct cobordisms between connected sums. We will then use connected sums to motivate and to briefly discuss how our results can be extended to higher dimensions.

We begin by discussing connected sums of Legendrian knots, as defined, for example, in \cite{etnyre-honda:connected-sum}. Given two Legendrian knots $\Lambda_1,\Lambda_2$, there is always a cobordism $\Lambda_1\cup\Lambda_2\prec\Lambda_1\#\Lambda_2$. Since Lagrangian cobordisms are not inherently symmetric, however, this cobordism cannot generally be reversed to get a cobordism $\Lambda_1\#\Lambda_2\prec\Lambda_1\cup\Lambda_2$.  

We do, however, have the following corollary of Theorem~\ref{thm:geom}.

\begin{corollary}
\label{cor:connect-sum}
Given Lagrangian cobordisms  $\Lambda_-^1\prec_{L^1}\Lambda_+^1$ and $\Lambda_-^2\prec_{L^2}\Lambda_+^2$ in $\R^3$, there is a Lagrangian cobordism
\[\Lambda_-^1\#\Lambda_-^2\prec_{L}\Lambda_+^1\#\Lambda_+^2.\]
\end{corollary}

The smooth connected sum is a satellite construction for $n=1$ strand.  We expect that the same must hold in the Legendrian case and that we can translate the construction to cobordisms. This is true up to one issue: a cobordism of knots $\Lambda_-\prec_{L}\Lambda_+$ is not necessarily a cobordism of patterns $\Pi_-\prec\Pi_+$, as the cobordism of patterns is required to be trivial on the boundary of the pattern, while a cobordism of knots is not required to be trivial anywhere.  We get around this issue by disentangling the satellite and cobordism operations.

\begin{lemma}
\label{lem:pre-connect-sum}
	Given a Lagrangian cobordism $\Lambda_- \prec_L \Lambda_+$ and a Legendrian knot $\Lambda'$, there exists a Lagrangian cobordism $\Lambda_- \# \Lambda' \prec_{L^\#} \Lambda_+ \# \Lambda'$.  
\end{lemma}

\begin{proof}
	We begin by setting up some language.	Recall from \cite[\S12]{chv} that a long Legendrian knot is a Legendrian embedding $\bbR \hookrightarrow (\bbR^3, \alpha_0)$ where the image agrees with the $x$ axis outside of a compact set.  On one hand, Chekanov proved that isotopy classes of long Legendrian knots are in one-to-one correspondence with isotopy classes of (closed) Legendrian knots \cite[Proposition 12.3]{chv}.  On the other hand, long Legendrian knots are also in correspondence with connected Legendrian $1$-tangles.  Further, the closure of a long Legendrian knot corresponds to taking the closure of the corresponding $1$-tangle to a solid torus knot and then taking the satellite around the maximal unknot.
	
	To construct the cobordism $L^\#$, we start by using the correspondences above to convert $\Lambda'$ into a $1$-tangle $\Pi'$.  Apply Theorem~\ref{thm:geom} to the cobordism $\Lambda_- \prec_L \Lambda_+$ and the trivial cobordism of tangles $\Pi' \prec \Pi'$ to obtain a cobordism $\Sigma(\Lambda_-, \Pi') \prec_{L^\#} \Sigma(\Lambda_+, \Pi')$.  Observe that for $n=1$ strand, the twist $\twist$ is trivial. It is clear that $\Sigma(\Lambda_\pm, \Pi')$ is  equivalent to $\Lambda_\pm \# \Lambda'$,  hence proving the lemma.
\end{proof}

The corollary will now follow quickly.  

\begin{proof}[Proof of Corollary~\ref{cor:connect-sum}]
	First apply Lemma~\ref{lem:pre-connect-sum} to $\Lambda_\pm = \Lambda_\pm^1$ and $\Lambda' = \Lambda_-^2$ to obtain a cobordism 
	\[\Lambda_-^1\#\Lambda_-^2=\Sigma(\Lambda_-^1,\Pi_-^2)\prec \Sigma(\Lambda_+^1,\Pi_-^2)=\Lambda_+^1\#\Lambda_-^2.\]
	Next, apply the lemma to $\Lambda_\pm = \Lambda_\pm^2$ and $\Lambda' = \Lambda_+^1$ to obtain a cobordism
	\[\Lambda_+^1\#\Lambda_-^2=\Sigma(\Lambda_-^2,\Pi_+^1)\prec \Sigma(\Lambda_+^2,\Pi_+^1)=\Lambda_+^1\#\Lambda_+^2.\]
	Finally, we compose the two cobordisms above to obtain the desired cobordism.
\end{proof}

\begin{remark}	
	There is also a version of this corollary when $L_1$ and $L_2$ are decomposable cobordisms.  The proof is trivial, though we could also mimic the proof above using Theorem~\ref{thm:decomp} in place of Theorem~\ref{thm:geom}.
\end{remark}

The corollary may be readily generalized to connected sums of Legendrians in higher dimensions.  In fact, the proof goes through almost word-for-word: first, it is straightforward to generalize the connect sum operation and Chekanov's long Legendrian knot material.  More importantly, the proof only uses Theorem~\ref{thm:geom} in the case of $1$-tangles.  This use is valid for the connected sum in higher dimensions, as the only places where the proof of Theorem~\ref{thm:geom} relied on three-dimensional techniques were in the description of twists at the Legendrian ends (which does not apply to a $1$-tangle) and the construction of the perturbing function $f$ in Section~\ref{ssec:push-off} (which is not necessary for a $1$-tangle).  

In order to fully generalize Theorem~\ref{thm:geom}, one would have to analyze singularities of a function $f: L \to \bbR$ of any index $1\leq k\leq n$. For each such singularity, one would have to find an appropriate version of a ``twist'' and then one would have to make rigorous the intuition discussed after Proposition~\ref{prop:construct-f} rather than relying on the explicit construction of $f$ in Section~\ref{ssec:push-off}, which uses the classification of surfaces.  We leave these developments to future work.


\bibliographystyle{amsplain}
\bibliography{main.bib}

\end{document}